\documentclass[a4paper,11pt]{article}

\usepackage[english]{babel}
\usepackage[english]{translator}
\usepackage[T1]{fontenc}
\usepackage[utf8]{inputenc}

\usepackage{amsmath}
\usepackage{amssymb}
\usepackage{amsfonts}
\usepackage{amstext}
\usepackage{amsthm}
\usepackage{bbm}

\usepackage{graphicx}
\usepackage{color}
\usepackage{psfrag}
\usepackage{subfigure}
\usepackage{float}

\usepackage{enumerate}
\usepackage{enumitem}
\usepackage{marvosym}
\usepackage{hyperref}

\usepackage[
  hmarginratio={1:1},
  vmarginratio={1:1},
  textwidth=460pt,
  heightrounded,
  bottom=3.65cm,
  top=3.65cm,
]{geometry}

\hypersetup{
  pdffitwindow=false,
  pdfhighlight=/O,
  pdfnewwindow,
  colorlinks=true,
  citecolor=red,
  linkcolor=blue,
  menucolor=blue,
  urlcolor=blue,
  pdfpagemode=UseOutlines,
  bookmarksnumbered=true,
  linktocpage,
  pdfkeywords={},
  pdfcreator={pdflatex},
  pdfproducer={LaTeX with hyperref}
}

\newcommand{\D}{\mathcal{D}}
\newcommand{\C}{\mathbb{C}}
\newcommand{\R}{\mathbb{R}}

\newcommand{\N}{\mathbb{N}}

\newcommand{\GL}{\mathrm{GL}}
\renewcommand{\Re}{\mathrm{Re}\,}
\renewcommand{\Im}{\mathrm{Im}\,}

\renewcommand{\L}{\mathcal{L}}

\newcommand{\diag}{\mathrm{diag}}

\DeclareMathOperator*{\esssup}{ess\,sup\,}

\renewcommand{\i}{\mathrm{i}} 
\renewcommand{\d}{\,\mathrm{d}}
\newcommand{\x}{\mathbf{x}}
\renewcommand{\H}{\mathcal{H}}
\newcommand{\W}{\mathcal{W}}
\newcommand{\I}{\mathcal{I}}
\renewcommand{\GL}{\text{\tiny GL}}
\newcommand{\Lo}{\text{\tiny Lo}}
\newcommand{\Lop}{\text{\tiny Lo+1}}

\newcommand{\ut}{\tilde{u}}

\newcommand{\quotes}[1]{``#1''}

\newcommand{\sect}
{
  \setcounter{equation}{0}
  \setcounter{figure}{0}
  \section
}

\theoremstyle{definition}
\newtheorem{definition}{Definition}[section]

\newtheorem{remark}[definition]{Remark}

\theoremstyle{plain}
\newtheorem{theorem}[definition]{Theorem}
\newtheorem{lemma}[definition]{Lemma}
\newtheorem{corollary}[definition]{Corollary}
\newtheorem{proposition}[definition]{Proposition}

\allowdisplaybreaks

\begin{document}
\title{Uniform $L^{\infty}$-bounds for energy-conserving higher-order time integrators for the Gross-Pitaevskii equation with rotation}
\setlength{\parindent}{0pt}

\begin{center}
{\Large Uniform $L^{\infty}$-bounds for energy-conserving \\ higher-order time integrators for the Gross-Pitaevskii equation with rotation} \\
\vspace{12pt}
Christian D{\"o}ding\footnotemark[1] and Patrick Henning\footnotemark[2] \\
\vspace{12pt}
October 4, 2022
\end{center}

\footnotetext[1]{Department of Mathematics, Ruhr-University Bochum, Universit\"atsstr. 150, 44801 Bochum, Germany, \\ e-mail: \textcolor{blue}{christian.doeding@rub.de}.}
\footnotetext[2]{Department of Mathematics, Ruhr-University Bochum, Universit\"atsstr. 150, 44801 Bochum, Germany, \\ e-mail: \textcolor{blue}{patrick.henning@rub.de}.}

\noindent
\begin{center}
\begin{minipage}{0.8\textwidth}
  {\small
    \textbf{Abstract.}
 In this paper, we consider an energy-conserving continuous Galerkin discretization of the Gross--Pitaevskii equation with a magnetic trapping potential and a stirring potential for angular momentum rotation. 
 The discretization is based on finite elements  in space and time and allows for arbitrary polynomial orders. It was first analyzed in [\emph{O.~Karakashian, C.~Makridakis; SIAM J. Numer. Anal.~36(6):1779--1807, 1999}] in the absence of potential terms and corresponding a priori error estimates were derived in $2D$. In this work we revisit the approach in the generalized setting of the  Gross--Pitaevskii equation with rotation and we prove uniform $L^{\infty}$-bounds for the corresponding numerical approximations in $2D$ and $3D$ without coupling conditions between the spatial mesh size and the time step size. With this result at hand, we are in particular able to extend the previous error estimates to the $3D$ setting while avoiding artificial CFL conditions.
}
\end{minipage}
\end{center}

\vspace{12pt}

\vspace{12pt}
 
\sect{Introduction}
When a dilute bosonic gas is cooled down near to the absolute zero temperature at 0 Kelvin, a so-called Bose-Einstein condensate (BEC) is formed \cite{Bose24,Einstein24,DMA95,AEM95}.  Such a BEC is an extreme state of matter which behaves, in its entity, like a macroscopic \quotes{super particle} and which hence allows to study quantum mechanical phenomena on observable scales. The central equation for mathematically modelling the dynamics of Bose--Einstein condensates is the Gross-Pitaevskii equation (GPE) \cite{Gross61,Lieb01,Pitaevskii61}. It seeks a scalar complex-valued wave function $u = u(\x,t) \in \C$ such that
	\begin{align} \label{GPE0}
		\i  \partial_t u & = -\Delta u + \i \mathbf{\Omega} \cdot ( \mathbf{x} \times \nabla) u + Vu + \beta |u|^2u
	\end{align} 
and together with appropriate initial and boundary conditions. Given are the real-valued function $V = V(\x) \in \R$, the vector $\mathbf{\Omega} \in \R^3$ and a real scalar $\beta$ which we assume to be positive in this work (i.e. $\beta \ge 0$). In the context of BECs, the solution $u$ to the GPE describes the quantum state of the condensate, $|u|^2$ is its (physically observable) density and the function $V$ has the role of a magnetic trapping potential that confines the BEC. Furthermore, the parameter $\beta$ encodes information about the number and the type of bosons. In particular, it characterizes if the interaction between the particles in the BEC are repulsive ($\beta > 0$) or attractive ($\beta < 0$). The term $\i \, \mathbf{\Omega} \cdot (\x \times \nabla)u$ models a stirring potential and hence describes an angular rotation of the BEC with angular velocity $\mathbf{\Omega} \in \R^3$. Taking the quantum mechanical momentum operator $\mathbf{P} = -\i \nabla$ into account as well as the angular momentum operator $\mathbf{L} = \x \times \mathbf{P} = -\i (\x \times \nabla)$ we can write $\i \, \mathbf{\Omega} \cdot (\x \times \nabla) = - \mathbf{\Omega} \cdot \mathbf{L}$. As a common simplification, we assume that the BEC rotates around the $z$-axis such that $\mathbf{L} = (0,0, \L_z)$ with $\L_z = -\i (x \partial_y - y \partial_x)$ and $\mathbf{\Omega} = (0,0,\Omega)$ for $\Omega \in \R$. In this case, the rotational term simplifies to $\i \, \mathbf{\Omega} \cdot (\x \times \nabla)u = - \Omega \, \L_z u$. From a physical perspective it is interesting to consider rotating Bose--Einstein condensates as such a configuration allows for the appearance of quantized vortices as a sign of the superfluid behavior of a BEC \cite{ARV01}. 

In the following we specify the precise initial-boundary-value problem for the GPE that we are considering in this work: Suppose $\D \subset \R^{d}$, $d = 1,2,3$, is a bounded domain and $I = [0,T) \subset \R$ a time interval. Then we consider the initial-boundary-value problem for $u: \overline{\D} \times I \rightarrow \C$ given by
	\begin{align} \label{GPE}
	\begin{alignedat}{2}
		\i  \partial_t u & = -\Delta u - \Omega \L_z u + Vu + \beta |u|^2u && \quad \text{in } \D \times I, \\
		u  & = 0 && \quad \text{on } \partial \D \times I, \\
		u & = u_0 && \quad \text{on } \D \times \{ t = 0 \}
	\end{alignedat}
	\end{align}
for a suitable initial value $u_0$. Due to the rotational term this problem only makes sense in the cases $d = 2,3$. Nevertheless, we include the one-dimensional case for which we assume that the rotational term is neglected. In particular, we formally set $\L_z = 0$ if $d = 1$.

There is a rich literature on the numerical treatment of \eqref{GPE} that is too comprehensive to discuss it in detail. Exemplarily, we refer to \cite{ABB13d,AST21,Bao14,BaJiMa03,CCW20,Lub08,ORS21,Tha12b,Zou01} and the references therein to get an overview over the field. One particularly important aspect is that the analytical equation \eqref{GPE} conserves the total energy of the system and it was numerically observed \cite{NLSComparison} that an analogous discrete energy conservation can be a crucial property of numerical schemes to get reliable approximations in practical situations. Time integrators that conserve a modified energy are for example the popular Besse relaxation scheme \cite{Bes04,BeDeDuLa21,Zouraris20} and the family of exponential Runge--Kutta schemes proposed in  \cite{FuHuZh22}. Time integrators that conserve the exact energy (up to spatial discretization errors) are more rare and include the Crank--Nicolson method based on the averaging of densities \cite{ADK91,BaC13,HeP17,HeWa21,San84} and the continuous Galerkin time stepping proposed in \cite{Makridakis99}. The latter method is also the numerical scheme that we shall consider in this paper, as it is not only energy-conservative, but it also allows for arbitrarily fast convergence rates for smooth solutions. This makes it very attractive for practical computations.

The method was first analyzed by Karakashian and Makridakis \cite{Makridakis99} who considered the GPE \eqref{GPE} for the case $\Omega = 0$ and  $V \equiv 0$ and in space dimension $d=2$. The authors proved well-posedness of the scheme together with a priori error estimates in the $L^{\infty}(L^2)$-norm and the $L^{\infty}(H^1)$-norm. If the spatial discretization remains unchanged during the time stepping, the error analysis was established under a coupling condition between the time step size $\tau$ and the spatial mesh size $h$ of the form $\log(h) \tau^{q-1} \rightarrow 0$ for $h$,$\tau \rightarrow 0$. Here, $q$ denotes the polynomial degree used for the time integration, which also shows that the analysis is not applicable to the lowest order case $q=1$. The coupling condition is a result of the proof technique, which requires bounding the growth of the nonlinear term $ |u|^2u$ by introducing a suitable truncation function. After that, the error between the exact solution and the truncated numerical approximation is analyzed and corresponding error estimates are derived. Finally, one needs to argue that the corresponding (truncated) numerical approximations remain uniformly bounded in $L^{\infty}(L^{\infty})$ independent of the truncation, so that estimates remain valid for the original scheme without truncation. To establish these uniform $L^{\infty}(L^{\infty})$-bounds, the authors used the error estimate for the truncated approximations together with an inverse inequality in finite element spaces. This causes the logarithmic coupling condition of the form $\log(h) \tau^{q-1} \rightarrow 0$, where $\log(h)$ enters through the inverse inequality. If the same strategy would be used for $d=3$, the inverse inequality would introduce a factor of $h^{-1/2}$ and the coupling condition would become $h^{-1/2} \tau^{q-1} \rightarrow 0$ as $h,\tau \rightarrow 0$.

In this paper we propose a different strategy to obtain the desired $L^{\infty}(L^{\infty})$-bounds both in $2D$ and $3D$ which does not introduce any coupling conditions. For that we apply a technique that was introduced by Li and Sun \cite{LiSu13} to remove mesh coupling conditions for a semi-implicit Euler discretization of nonlinear parabolic problems. The idea of that technique is to split the error based on a semi-discrete auxiliary problem (which is analytical in space and discrete in time) in order to separate the spatial and temporal discretization and to only apply inverse inequalities on terms that purely depend on spatial errors. The first application of this technique to nonlinear Schr\"odinger equations such as the GPE was established by Wang \cite{Wan14} who considered an Adams--Bashforth-type linearization of the Crank-Nicolson method. Applications to the energy-conservative Crank--Nicolson method were developed in \cite{HeP17,HeWa21}. In all cases, lowest order finite element spaces were considered. In fact, one of the major drawbacks of the technique is that it is not directly applicable to higher order finite element spaces as considered in this work. The reason is that this would also require higher order regularity of the solutions to the auxiliary problem and also corresponding stability bounds in higher order Sobolev norms. Such stability estimates for the auxiliary problem are however only available in $H^2(\D)$, which would hence only yield optimale convergence rates for $\mathcal{P}^1$-Lagrange finite elements or for particular generalized FE spaces that exploit only low regularity \cite{HeW22}. In this work we solve this issue by only applying the error splitting technique to obtain $L^{\infty}(L^{\infty})$-bounds for the discrete approximations, as this does not require higher regularity of the semi-discrete auxiliary solution. Once the $L^{\infty}(L^{\infty})$-bounds are available, the optimal rates for higher order finite element spaces can be obtained with the same strategy as in the work by Karakashian and Makridakis \cite{Makridakis99}.

We finish the introduction with an outline of this paper. In the following section \ref{sectionpreliminaries} we state our first assumptions on \eqref{GPE0} and collect important properties of the rotating GPE after introducing our basic notation. Then we proceed in section \ref{sectionnumericalmethod} with describing our spatial and temporal discretization for problem \eqref{GPE0} and we formulate the fully-discrete method that we are considering in this work. Furthermore, we prove that the numerical method is energy-conserving in time and then state our main result concerning the uniform $L^\infty(L^\infty)$-bounds for the numerical approximations. From there we conclude the a priori error estimates w.r.t. $L^\infty(L^2)$ and $L^\infty(H^1)$ as in \cite{Makridakis99}. In section \ref{section:reformulation} we give a reformulation of the method which, on the one hand, is used in the further error analysis and, on the other hand, gives access to a simple implementation. Section \ref{sec:truncetedproblem} is then devoted to a fully-discrete truncated problem where the cubic nonlinearity is replaced by the aforementioned truncated nonlinearity. We prove well-posedness of this problem and continue with the corresponding semi-discrete truncated problem in section \ref{sec:errorsemi}. For the semi-discrete problem error estimates are derived w.r.t. to $L^\infty(H^1)$, $L^\infty(H^2)$ and $L^\infty(L^\infty)$. In section \ref{section:proof-main-result} we prove our main result. There we start by showing error estimates of the fully-discrete problem to the projection of the semi-discrete problem onto the finite element space. After that, the error splitting is used to prove the main result and the section is closed with a uniqueness result for the fully-discrete approximation. Finally, in section \ref{sec:numerics} we test the performance of the method in numerical experiments.

\sect{Preliminaries}
\label{sectionpreliminaries}

In this section we work out some preliminaries regarding the GPE and state our first assumptions on the problem. Before doing so, we introduce our basic notation: With $\x = (x,y,z)$ (resp. $\x = (x,y)$ and $\x = x$) we denote the spatial coordinate in $\D \subset \R^3$ (resp. $\D \subset \R^2$ and $\D \subset \R$). For a complex number $x \in \C$ we write $\overline{x}$ for its complex conjugate and $\x \cdot \mathbf{y} = \sum_{i = 1}^n \overline{x_i} y_i$ is the scalar product on $\C^n$ such that $|\x| = \sqrt{\x \cdot \x}$ is a norm. For measurable functions with values in the complex plane $\C$ we define the Lebesgue spaces $L^p(\D,\C)$, $1 \le p \le \infty$ in the usual sense and denote by $\| u \|_{L^p(\D)}$ the norm in $L^p(\D,\C)$. Usually we neglect the image space when its meaning is clear from the context, i.e., $L^p(\D,\C) = L^p(\D)$. The inner product on the complex Hilbert space $L^2(\D,\C)$ is defined by $(u,v) = \int_\D \overline{u} \,v \d\x$ and
we denote by $W^{k,p}(\D,\C) = W^{k,p}(\D)$ the standard Sobolev spaces on $\D$ with norm and semi-norm given by
\begin{align*}
	\| u \|_{W^{k,p}(\D)} = \Big( \sum_{|\alpha| \le k} \| D^\alpha u \|_{L^p(\D)}^p \Big)^{1/p}, \quad | u |_{W^{k,p}(\D)}^2 = \Big( \sum_{|\alpha| = k} \| D^\alpha u \|_{L^p(\D)}^p \Big)^{1/p}
\end{align*}
and with $\| u \|_{W^{k,\infty(\D)}} = \max_{|\alpha| \le k} \| D^\alpha u \|_{L^\infty(\D)}$ in the case $p = \infty$. If $p = 2$ we also write $H^{k}(\D) = W^{k,2}(\D)$ and denote by $H^1_0(\D,\C) = H^1_0(\D)$ the space of functions $u \in H^1(\D)$ with zero trace. Furthermore, $H^{-1}(\D,\C) = H^{-1}(\D)$ is defined as the dual of $H^1_0(\D)$. For Bochner-measurable functions with values in a Banach space $X$ we define the Bochner-Lebesgue spaces $L^p(I,X) = \{ u : I \rightarrow X \text{ Bochner-measurable}: \| u \|_{L^p(I,X)} < \infty\}$ for $1 \le p \le \infty$ and where
\begin{align*}
	& \| u \|_{L^p(I,X)}:= \Big( \int_I \| u(t) \|_X^p \d t \Big)^{1/p}, \quad 1 \le p < \infty, \\
	& \| u \|_{L^\infty(I,X)} : = \esssup_{t \in I} \| u(t) \|_X.
\end{align*}
Finally, we define the Bochner-Sobolev spaces $W^{k,p}(I,X) = \{ u \in L^p(I,X): \partial_t^j u \in L^p(I,X),\, j = 1,\dots,k\}$ in the usual sense (see e.g. \cite{Bartels15} and \cite{Cazenave03}) and equip them with the usual norms denoted by $\| \cdot \|_{W^{k,p}(I,X)}$. \\

Now we define the notion of a weak solution of the Gross-Pitaevskii problem in a rotating frame: 

\begin{definition}[weak solution of the GPE] Let $u_0 \in H^1_0(\D)$ be a given initial value. Then $u \in L^\infty(I,H^1_0(\D))$ with $\partial_t u \in L^\infty(I,H^{-1}(\D))$ is a weak solution of the GPE, if $u(\cdot,0) = u_0$ and
	\begin{equation} \label{weakGPE}
 		(\i \partial_t u,  v) = (\nabla u,  \nabla v) - \Omega(\L_z u,v) + (Vu,v) + \beta (|u|^2 u,v)
	\end{equation} 
holds for all $v \in H^1_0(\D)$ and almost every $t \in I$.
\end{definition}

Note that if $u$ is a weak solution of the GPE then $u \in C(I,L^2(\D))$ and the notion of the initial value $u(\cdot,0) = u_0$ is well-defined. Next we collect our basic assumptions on the system \eqref{weakGPE} which guarantee well-posedness in particular cases and which are necessary for our further analysis:
\begin{enumerate}[label={(A\arabic*)}]
\item \label{A1} The spatial domain $\D \subset \R^d$, $d = 1,2,3$, is convex, bounded and polyhedral. The time interval is given by $I = [0,T)$ for some $T > 0$.
\item \label{A2} The interaction parameter $\beta$ is real and positive, i.e., $\beta \ge 0$.
\item \label{A3} The potential $V$ satisfies $V \in L^\infty(\D,\R)$ and $V(\x) \ge 0$ for almost all $\x \in \D$.
\item \label{A4} The rotation velocity $\Omega$ is real and there is $\lambda > 0$ such that
\begin{align*}
	V(\x) - \frac{1 + \lambda}{4} \Omega^2 (x^2 + y^2) \ge  0 \quad \text{for almost all } \x \in \D.
\end{align*}
\end{enumerate}

Assumption \ref{A4} states that trapping frequencies of the potential $V$ (in $x$- and $y$-direction) are sufficiently large compared to the rotation frequency $\Omega$. Physically speaking, this ensures that the centrifugal forces cannot become strong enough to destroy the condensate. An equivalent condition can be found in \cite{BWM05}, where it is shown that ground states of BECs exist if \ref{A4} is fulfilled and that they do not exist if $V(\x) - \tfrac{1}{4} \Omega^2 (x^2 + y^2)<0$ for harmonic trapping potentials of the form $V(\x)=\tfrac{1}{2}(\gamma_x^2 x^2 + \gamma_y^2 y^2)$. The case $V(\x) = \tfrac{1}{4} \Omega^2 (x^2 + y^2)$ is a borderline case, where the existence of stable BECs is unclear. \\

Well-posedness of the Cauchy problem \eqref{weakGPE} on bounded domains has been investigated in the classical textbook by Cazenave \cite{Cazenave03} for the GPE without rotation ($\Omega = 0$). In this case there exists a maximal (maybe infinite) time $T = T(u_0, \D) > 0$ such that \eqref{weakGPE} admits a (strong) solution $u$ satisfying $u \in C(I,H^1_0(\D))$ and $\partial_t u \in C(I,H^{-1}(\D))$. In addition, this solution is unique in one and two space dimensions. For the GPE with rotation ($\Omega \neq 0$) much less is known: In \cite{Antonelli12} well-posedness was shown in two and three space dimensions for the case $\D = \mathbb{R}^d$. An earlier work concerning the case $\D = \R^3$ with more restrictive conditions on $\Omega$ and $V$ is \cite{Hao07}. However, to the best of our knowledge well-posedness on bounded domains for the GPE with rotation is still open in the literature. \\

Equation \eqref{weakGPE} can be formally derived from the energy functional of the system given by
	\begin{align} \label{energy}
		E(u) := \frac{1}{2} \int_{\D} |\nabla u|^2 - \Omega \, \overline{u} \, \L_z u + V |u|^2 + \frac{\beta}{2}|u|^4 \d\x, \quad u \in H^1_0(\D).
	\end{align}
If we consider (for the moment) $L^2(\D)=L^2(\D,\C)$ as a \emph{real} Hilbert space with (real) scalar product $(u,v)_\Re : = \Re \int_{\D} \overline{u} v \d \x$, then the energy functional $E$ is Fr\'echet differentiable on the \emph{real} Hilbert space $H^1_0(\D)$ with derivative $E' : H^1_0(\D) \rightarrow H^{-1}(\D)$ given by
\begin{align*}
	\langle E'(u), v \rangle_\Re = \Re \big( (\nabla u,  \nabla v) - \Omega (\L_z u,v) + (Vu, v) + \beta (|u|^2 u,v) \big).
\end{align*}
Here $\langle \cdot, \cdot \rangle_\Re$ denotes the dual pairing between the real Hilbert spaces $H^{-1}(\D)$ and $H^1_0(\D)$. With this, \eqref{weakGPE} can be written as the Hamiltonian system
\begin{align} \label{symplectic}
	\omega ( \partial_t u, v)  = \langle E'(u), v \rangle_\Re \quad \forall v \in H^1_0(\D)
\end{align}
with the symplectic form $\omega: H^1_0(\D) \times H^1_0(\D) \rightarrow \R$ given by $\omega(v,w) := (\i v, w)_\Re$. The system \eqref{symplectic} is symplectic which means that the flow of \eqref{symplectic} preserves the symplectic form $\omega$ (see e.g. \cite{Hairer18}). In addition, when testing \eqref{symplectic} with $\partial_t u$ it is easily seen that the energy \eqref{energy} is conserved over time. Furthermore, when we test \eqref{symplectic} with $\i u$ it follows that the mass of the system
given by $\| u(t) \|_{L^2(\D)}^2$ is also conserved over time. From the numerical point of view, it can be extremely important to use integrators that preserve one or more of the above mentioned structures; i.e. energy, mass or the symplectic structure of the system. However, it is well-known that in the general nonlinear case there exists no time integrator (that is not exact) that conserves both the energy and the symplectic structure of the system simultaneously; see e.g. \cite{Sanz-Serna88} and \cite{Marsden88}. The Crank-Nicolson time integrator that was used in \cite{HenningMalqvist17} conserves both energy and mass  and turned out to be successful in simulating the dynamics of rotating BECs. Nevertheless, high order integrators (order $p > 2$) typically cannot conserve both energy and mass of the system at the same time. Numerical experiments indicate that, in practical situations, the conservation of energy should be prioritized over the conservation of mass, cf. \cite{NLSComparison}. As we will see in section \ref{sectionnumericalmethod}, the high-order time integrator that we consider in this work is in fact energy-conserving. \\[0.5em]
Next we introduce the conjugate-linear operator
\begin{align*}
	\H: H^1_0(\D) \rightarrow H^{-1}(\D), \quad \langle \H u, v \rangle = (\nabla u, \nabla v ) - \Omega( \L_z u, v ) + (Vu, v)
\end{align*}
where from now on $H^1_0(\D)$ and $H^{-1}(\D)$ are again complex Hilbert spaces with the complex dual pairing $\langle \cdot, \cdot \rangle$. Note that for $u \in H^1_0(\D) \cap H^2(\D)$, we can identify $\H u$ uniquely with an $L^2(\D)$-function, where $\H u = - \Delta u - \Omega \, \L_z u + V u$.
As in \cite{HenningMalqvist17}, we will show that the operator $\H$ induces the sesquilinear form
\begin{align*}
	(u,v)_\H := \int_\D \overline{(\nabla u - \tfrac{\i}{2}b u)} \cdot (\nabla v - \tfrac{\i}{2}b v) + (V - \tfrac{1}{4}|b|^2 )\overline{u} v \d \mathbf{x}, \quad b(\mathbf{x}) = \begin{cases} 0, & d = 1, \\ \Omega (y,-x)^\top, & d = 2, \\
	\Omega (y,-x,0)^\top, & d = 3. \end{cases}
\end{align*}
In particular, we have the following lemma:
\begin{lemma} \label{elliptic}
Under the assumptions \ref{A1}, \ref{A3} and \ref{A4} the sesquilinear form $(\cdot,\cdot)_\H$ defines a scalar product on $H^1(\D)$ and the induced norm $\| \cdot \|_\H$ is equivalent to the standard $H^1$-norm, i.e. for $\| v \|_\H = \sqrt{(v,v)_\H}$ there are $C_1,C_2 > 0$ such 	that
	\begin{align*}
		C_1 \| v \|_{H^1(\D)} \le \| v \|_{\H} \le C_2 \| v \|_{H^1(\D)} \quad \forall v \in 						H^1(\D).
	\end{align*}
Moreover, $\H$ is continuous, uniformly elliptic on $H^1_0(\D)$ and the following statements hold:
\begin{enumerate}
	\item[i)] $(u,v)_\H = \langle \H u, v \rangle$ for all $u,v \in H^1_0(\D)$.
	\item[ii)] $\| v \|_{H^2(\D)} \le C \| \H v \|_{L^2(\D)}$ for all $v \in H^1_0(\D) \cap H^2(\D)$ and some $C > 0$. 
\end{enumerate}
\end{lemma}

The proof of the lemma is done in the appendix since it is just a slight modification of a similar result given in \cite{HenningMalqvist17}.

\begin{remark}
The proof of Lemma \ref{elliptic} in the appendix shows that the ellipticity constant $C_1$ is given by $C_1^2 = \lambda (1 + C_P^2)(1 + \lambda)^{-1}$ where $\lambda>0$ is the constant appearing in \ref{A4} and $C_P > 0$ is the constant of the Poincar\'e -Friedrichs inequality $\| v \|_{L^2(\D)} \le C_P \| \nabla v \|_{L^2(\D)}$ for all $v \in H^1_0(\D)$. In particular, it holds $C_1 = \mathcal{O}(\lambda^{1/2})$ and hence the operator $\H$ degenerates in the case $\lambda = 0$, which is consistent with the aforementioned physical interpretation of dominating centrifugal forces.
\end{remark}

\sect{Numerical discretization and main result} \label{sectionnumericalmethod}

In this section we formulate the numerical scheme to approximate the solution of \eqref{GPE}. We start with the space discretization. \\

\textit{Space discretization.} Let $S_h \subset H^1_0(\D)$ be a finite dimensional subspace parametrized by a mesh size parameter $h > 0$. Our further assumptions on the space discretization are made implicitly in terms of properties of the Ritz-projection $P_h: H^1_0(\D) \rightarrow S_h$ with respect to the scalar product $(\cdot,\cdot)_{\H}$. To be precise, for $v \in H^1_0(\D)$ the image $P_hv \in S_h $ fulfills
\begin{align*}
	(v - P_h v,w)_{\H} = 0 \quad \forall w \in S_h. 	
\end{align*}
We pose the following assumptions on the space discretization that are given by the approximation properties of $P_h$ and the existence of an inverse estimate on $S_h$:
\begin{enumerate}[label={(A\arabic*)}]
\setcounter{enumi}{4}
\item \label{A5} There are $r \in \N$ and $C,h_0 > 0$ such that for $k = 0,1$ and for all $0 < h \le h_0$ we have the error estimate
\begin{align*}
	\| v - P_h v \|_{H^k(\D)} \le Ch^{s-k} |v|_{H^s(\D)} \quad \forall v \in H^s(\D) \cap H^1_0(\D) \quad 2 \le s \le r+1. 
\end{align*}
\item \label{A6} There are $C,h_0 > 0$ such that for all $0 < h \le h_0$ we have the inverse estimate
\begin{align*}
	\| v_h \|_{L^\infty(\D)} \le C h^{-\frac{d}{2}} \|v_h \|_{L^2(\D)} \quad \forall v_h \in S_h.
\end{align*}
\item \label{A7} There are $C_\infty,h_0 > 0$ such that for all $0 < h \le h_0$ we have the error estimate
\begin{align*}
	\| v - P_h v \|_{L^\infty(\D)} \le C_\infty |v|_{H^2(\D)} \quad \forall v \in H^2(\D) \cap H^1_0(\D).
\end{align*}
\end{enumerate}
An admissible choice for the space discretization satisfying our assumptions is the space of standard $H^1$-conforming $\mathcal{P}^{r}$-Lagrange finite elements on a quasi-uniform mesh. In that case, \ref{A5} is a standard error bound and \ref{A6} is a standard inverse estimate; see \cite{BrennerScott08}. To obtain \ref{A7}, one may split the error into
\begin{align*}
	\| v - P_h v \|_{L^\infty(\D)} \le \| v - \mathcal{I}_h v \|_{L^\infty(\D)} + \| \mathcal{I}_h v - P_h v \|_{L^\infty(\D)}
\end{align*}
where $\mathcal{I}_h: C(\overline{\D}) \rightarrow S_h$ is the conventional Lagrange (nodal) interpolation operator. The first term is of order $\mathcal{O}(h^{2 - \frac{d}{2}})$ if $u \in  H^1_0(\D) \cap H^2(\D)$ whereas for the second term we can use the inverse estimate $\| \mathcal{I}_h v - P_h v \|_{L^\infty(\D)} \le Ch^{1 - \frac{d}{2}}\| \mathcal{I}_h v - P_h v \|_{H^1(\D)} \le Ch^{1 - \frac{d}{2}}(\| \mathcal{I}_h v - v \|_{H^1(\D)} + \|v - P_h v \|_{H^1(\D)})$. Now we can use again estimates for the nodal interpolation operator and \ref{A5} to see that the resulting term is of order $\mathcal{O}(h^{2 - \frac{d}{2}})$ if $u \in H^1_0(\D) \cap H^2(\D)$. Hence \ref{A7} is fulfilled. \\

We continue with the time discretization. \\

\textit{Time discretization.} For the time discretization we choose a quasi-uniform partition $0 = t_0 < t_1 < \cdots < t_N = T$ of $\overline{I} = [0,T]$ into $N$ subintervals and for the $n$'th time step we define
\begin{align*}
	I_n = (t_{n},t_{n+1}], \quad \tau_n = t_{n+1} - t_{n}, \quad n = 0,\dots,N-1.
\end{align*}
The discretization parameter $\tau$ is defined as the maximum step size, that is
\begin{align*}
	\tau = \max_{n = 0,\dots,N-1} \tau_n.
\end{align*}
The quasi-uniformity of the partition means that there is a $\tau$-independent constant $\rho_0 > 0$ such that $\rho_0 \tau \le \min \{ \tau_n : \, n = 0,\dots,N-1\}$ uniformly 
for all admissible partitions of $\bar{I}$.
The time discretization we use is based on a continuous Galerkin ansatz, where the resulting system is resolved with sufficiently high order Gauss quadrature rules which can be used in a practical implementation of the numerical scheme. In order to formulate the numerical scheme, we start by introducing interpolation spaces in time with pointwise values in the finite element space $S_h$: For a polynomial degree $q \in \N$ and time indices $n = 0, \dots N-1$ we define
\begin{align*}
	& V_q := \big\{ v : \D \times (0,T] \rightarrow \C\, : \, v_{| \D \times I_n} (\x,t) = \sum_{j = 0}^q t^j \, \varphi_{n j} (\x),\, \mbox{where }\varphi_{nj} \in S_h \big\}, \\
	& V^n_{q} := \{ v_{| \D \times I_n}: v \in V_q\}.
\end{align*}
Now we test \eqref{weakGPE} pointwise with $v = v(t)$, $v \in V_q$ and integrate in time over $I_n$. Then we obtain the following fully-discrete numerical scheme:

\begin{definition}[Fully-discrete cG($q$)-scheme for GPE]
Assume \ref{A1}-\ref{A7} and let $u_0 \in H^1_0(\D)$ be given. Then the fully-discrete approximation is given by a function $u_{h,\tau} \in V_q$ such that for all $n = 0,\dots, N-1$
\begin{align} \label{cGscheme}
\begin{split}
	\int_{I_n} (\i \partial_t u_{h,\tau}, v) - (u_{h,\tau}, v)_{\H} - \beta (|u_{h,\tau}|^2 u_{h,\tau},v) \d t & = 0 \quad \forall v \in V_{q-1}, \\
	u_{h,\tau}^{n+} & = u_{h,\tau}(t_{n}),
\end{split}
\end{align}
where $u_{h,\tau}^{n+} = \lim_{t \searrow t_n} u_{h,\tau}(t)$ and $u_{h,\tau}^{0+} = P_h u_0$.
\end{definition}

Note that the second equation in \eqref{cGscheme}, i.e., the jump condition, implies continuity of $u_{h,\tau}$ in $t$ over the whole interval $(0,T]$. Therefore we can extend $u_{h,\tau} \in V_q$ to a continuous function on $\overline{I} = [0,T]$ by setting $u_{h,\tau}(t_0) = P_h u_0$. On the other hand, the global continuity condition already determines one degree of freedom of $u_{h,\tau}$ on each time interval $I_n$. Therefore, the space of test functions has one degree less than the space of ansatz functions. This type of methods is also known as Petrov-Galerkin methods. As we will prove now this is sufficient to preserve the energy of the fully-discrete approximation.

\begin{proposition}[Discrete energy conservation] \label{energyconservation}
Let the assumptions \ref{A1}-\ref{A7} be satisfied and let $u_{h,\tau} \in V_q$ be a fully-discrete approximation of $u$ satisfying \eqref{cGscheme} for some $u_0 \in H^1_0(\D)$. Then $u_{h,\tau}$ preserves the energy at the time instances $t_n$, i.e., for all $n = 1,\dots,N$ we have
\begin{align*}
	E( u_{h,\tau}(t_n)) = E( u_{h,\tau}(t_0)).
\end{align*}  
\end{proposition}

\begin{proof}
We test equation \eqref{cGscheme} with $\partial_t u_{h,\tau} \in V_{q-1}$ and take the real part to obtain
\begin{align*}
	& E( u_{h,\tau}(t_{n+1})) - E( u_{h,\tau}(t_n)) = \int_{I_n} \frac{\d}{\d t} E(u_{h,\tau}) \d t \\
	& = \Re \int_{I_n} (u_{h,\tau}, \partial_t u_{h,\tau} )_\H + \beta (|u_{h,\tau}|^2 u_{h,\tau} ,\partial_t u_{h,\tau}) \d t = - \Re \int_{I_n} \i |\partial_t u_{h,\tau}|^2 \d t = 0
\end{align*}
for $n = 0,\dots,N-1$. This proves the claim.
\end{proof}

Now we aim to formulate our main result of this work for which we make the assumption that there exists a smooth solution $u$ of \eqref{weakGPE} in the following sense:
\begin{enumerate}[label={(A\arabic*)}]
\setcounter{enumi}{7}
\item \label{A8} The GPE problem \eqref{weakGPE} has a solution $u$ satisfying
\begin{enumerate}[label=\textit{\roman*)}]
\item $u \in W^{q+2,\infty} (I,H^2(\D) \cap H^1_0(\D)) \cap W^{q+1,\infty}(I,W^{1,\infty}(\D))$,
\item $\H u \in W^{q+2,\infty} (I,H^1_0(\D))$.
\end{enumerate}
where  $q \in \N$ is the polynomial degree used for the time discretization.
\end{enumerate}

Observe that assumption \ref{A8} requires regularity of the initial value which is at least $u_0 \in H^2(\D) \cap H^1_0(\D)$. Furthermore, analogously to \cite[Lemma 3.1]{HeP17} it can be proved that any solution $u$ of \eqref{weakGPE} that fulfills the regularity assumptions \ref{A8} must be unique. \\

We continue with the main result of this work.
\begin{theorem} \label{mainresult}
Assume that the assumptions \ref{A1}-\ref{A8} are fulfilled. Then there are constants $M = M(u) > 0$ and $\tau_0,h_0 > 0$ such that for all $0 < \tau < \tau_0$ and $0 < h < h_0$ there exists a numerical approximation $u_{h,\tau}$ satisfying \eqref{cGscheme} and that fulfills the uniform bound
\begin{align*}
	\| u_{h,\tau} \|_{L^\infty(I \times \D)} < M.
\end{align*}
In particular, $M$ depending on $u$ can be chosen as
\begin{align} \label{M}
	M = \| u \|_{L^\infty(I \times \D)} + C_\infty \| u \|_{L^\infty(I,H^2(\D))} + 1
\end{align}
where $C_\infty$ is from \ref{A7}.
\end{theorem}
The proof of Theorem \ref{mainresult} is given in section \ref{section:proof-main-result}.\\[0.5em]
Form the above stated boundedness of the numerical solution we can conclude a priori error estimates for the approximation $u_{h,\tau}$ which generalize the results of \cite{Makridakis99}. To be precise, the a priori estimates shown in \cite{Makridakis99} only hold for the case $d=2$, $V = 0$, $\Omega = 0$ and under coupling conditions for the spatial and temporal mesh, i.e., that the time step size $\tau$ and the spatial parameter $h$ are such that $\log(h)\tau^{q-1} \rightarrow 0$ as $\tau,h \rightarrow 0$. The mesh condition in \cite{Makridakis99} is needed to show uniform boundedness of the approximations $u_{h,\tau}$ through an inverse estimate of the form $\| v \|_{L^\infty(\D)} \le |\log(h)| \| \nabla v \|_{L^2(\D)}$ for $v \in S_h$ and $h < 1$ (cf. \cite[Lemma 6.4]{Thomee97}) which is only valid in two dimensions. In three dimension one may use the inverse estimate $\| v \|_{L^\infty(\D)} \le Ch^{-1/2} \| \nabla v\|_{L^2(\D)}$ (cf. \cite[Lemma 1.142]{Ern04}) and the coupling condition would become $h^{-1/2} \tau^{q-1} \rightarrow 0$ as $h,\tau \rightarrow 0$. The boundedness of the approximation is essential to conclude the final error estimates. Thanks to our result (Theorem \ref{mainresult}) this mesh condition as well as the restriction to $d=2$ can be avoided to show uniform $L^{\infty}$-bounds and hence the desired error estimates. The generalization of the $L^\infty(L^2)$- and $L^\infty(H^1)$-error estimates to the case $V \neq 0$, $\Omega \neq 0$ is then straightforward.

\begin{corollary} \label{aprioriestimates}
Let the assumptions of Theorem \ref{mainresult} be fulfilled. If $u,\partial_t u \in L^\infty(I,H^{r+1}(\D))$ then there are constants $C = C(u) > 0$ and $\tau_0,h_0 > 0$ such that for all $0 < \tau < \tau_0$ and $0 < h < h_0$ the following a priori error estimates hold:
\begin{align*}
	& \| u - u_{h,\tau} \|_{L^\infty(I,L^2(\D))} \le C(\tau^{q+1} + h^{r+1}), \\
	& \| u - u_{h,\tau} \|_{L^\infty(I,H^1(\D))} \le C(\tau^{q+1} + h^{r}).
\end{align*}
Moreover, if $u \in C^{4q}(\overline{I} \times \overline{\D})$ and $V = 0$, $\Omega = 0$ then we have superconvergence at any $t_n$ with
\begin{align*}
	\max_{0 \le n \le N} \| u(t_n) - u_{h,\tau}(t_n) \|_{L^2(\D)} \le C(\tau^{2q} + h^{r+1}).
\end{align*}

\end{corollary}

The proof of Corollary \ref{aprioriestimates} follows analogously along the lines of \cite{Makridakis99}. Exemplarily, the proof of the $L^\infty(L^2)$-error estimates is given in the appendix. We point out, that superconvergence of order $\mathcal{O}(\tau^{2q})$ at the discrete points $t_n$ requires $V = 0$ and $\Omega = 0$. The generalization to the rotating GPE with $V \neq 0$ and $\Omega \neq 0$ is much more delicate. This is due to the fact, that the crucial step in proving the superconvergence is to show that $\H^s u = 0$ on $\partial \D$ for $s = 0,\dots,2q$ assuming $u \in C^{4q}(I \times \overline{\D})$. In \cite{Makridakis99} and before in \cite{Karakashian93} this was done for the case $\H = -\Delta$ by explicitly analyzing the second partial derivatives at the boundary. However, this procedure cannot be trivially extended to the general case $\H = -\Delta - \Omega \L_z + V$ and it is unclear if the strategy works out in that case. However, in applications one usually observes the superconvergence even in the presence of an angular momentum and a trapping potential. This is due to the fact that physically relevant solutions decay exponentially fast near the boundary. The same holds true for the derivatives of the solution so that $\H^s u$ is approximately zero on $\partial \D$. Therefore, possible low regularity close to the boundary produces only errors that are negligible compared to the (interior) errors caused by the time and space discretization.

\sect{Reformulation and implementation of the method}
\label{section:reformulation}

Before we turn our attention to the proof of our main result, we need to introduce a suitable reformulation of the  fully-discrete scheme \eqref{cGscheme}. The reformulated system gives not only access to a simple implementation, but it will be also an important ingredient of our error analysis. In particular, we explain in the following how the time integrals in \eqref{cGscheme} can be solved exactly using sufficiently high quadrature rules since the solution $u_{h,\tau}$ as well as the test function $v$ are polynomials in time on each $I_n$. For that purpose, we consider the $q$-stage Gauss-Legendre quadrature rule on the unit interval $[0,1]$, that is
\begin{align} \label{GLquadrature}
	\int_0^1 p(\tau) \d\tau = \sum_{j=1}^q w^\GL_{j} \, p(s^\GL_j) \quad \forall p \in \mathcal{P}^{2q-1}([0,1])
\end{align}
with the Gaussian quadrature nodes $0 < s^\GL_1 < \cdots < s^\GL_q < 1$ and weights $\{ w^\GL_j\}_{i = 1,\dots ,q}$. This Gauss-Legendre quadrature rule is exact for polynomials of degree $2q-1$ which is sufficient to integrate the linear terms in \eqref{cGscheme} exactly. In addition, we introduce the Lagrange polynomials of degree $q-1$ associated with the Gaussian quadrature nodes $\{ s^\GL_j\}_{j = 1,\dots,q}$ by
\begin{align*}
	\ell_i(s) = \prod_{j = 1, j \neq i}^q \frac{(s - s^\GL_j)}{(s^\GL_i - s^\GL_j)}, \quad i = 1,\dots,q.
\end{align*}
The polynomials $\ell_i$ will be later used to express the space of test functions. To express the space of ansatz functions (which has one degree more), we add $s^\GL_0 = 0$ to the collection of the Gaussian quadrature nodes such that $0 =s^\GL_0 < s^\GL_1 < \dots < s^\GL_q$, and denote the corresponding Lagrange polynomials of degree $q$ by
\begin{align*}
	\hat{\ell}_i(s) = \prod_{j = 0,  j \neq i}^q \frac{(s - s^\GL_j)}{(s^\GL_i - s^\GL_j)}, \quad i = 0,\dots,q.
\end{align*}
The unit interval $[0,1]$ is transformed to $\overline{I}_n$ by the transformation $s \mapsto t_n + s\, \tau_n$. Hence, we obtain the Gaussian quadrature nodes and weights on $\overline{I_n}$ via
\begin{align*}
	t^\GL_{n,j} := t_n + s^\GL_j \tau_n, \quad w^\GL_{n,j} := \tau_n \, w^\GL_j, \quad j = 1,\dots, q
\end{align*}
so that $\int\limits_{I_n} p(s) \d s = \sum\limits_{j = 1}^q w^\GL_{n,j} \, p(t^\GL_{n,j})$ for all polynomials $p$ on $I_n$ of degree less or equal $2q - 1$. Furthermore, we formally add the boundary points of $\overline{I}_n$ to the collection of the Gaussian quadrature nodes and denote them by
\begin{align*}
	t^\GL_{n,0} := t_n, \quad t^\GL_{n,q+1} := t_{n+1}.
\end{align*}
Next we transform the Lagrange polynomials $\ell_{i}$ of degree $q-1$ and $\hat{\ell}_{i}$ of degree $q$ to the interval $\overline{I}_n$ as well and define
\begin{align*}
& \ell_{n,i}(t_n + s \tau_n) := \ell_{i}(s), \quad i = 1,\dots,q,\\
& \hat{\ell}_{n,i}(t_n + s \tau_n) := \hat{\ell}_{i}(s), \quad i = 0,\dots,q.
\end{align*}
Since $u_{h,\tau} \in V_q$ we can express it uniquely on $I_n$ as
\begin{align} \label{ansatz}
	u_{h,\tau}(t) = \sum_{j = 0}^q \hat{\ell}_{n,j}(t) \, u^{n,j}_{h,\tau}, \quad \mbox{for } t \in I_n,
\end{align}
where $u^{n,j}_{h,\tau}  := u_{h,\tau}(t^\GL_{n,j}) \in S_h$. Testing in \eqref{cGscheme} with  $v(t) = \ell_{n,i}(t)\,\psi$ for $i = 1, \dots, q$ and arbitrary $\psi \in S_h$ yields that \eqref{cGscheme} is equivalent to
\begin{align} \label{cGscheme2}
\begin{split}
	 \sum_{j = 0}^q m_{ij} \, (u^{n,j}_{h,\tau}, \psi)  - \i \, \tau_n \, w^\GL_i \, (u^{n,i}_{h,\tau}, \psi )_{\H} & - \i \beta \int_{I_n} \ell_{n,i}(t) \, (|u_{h,\tau}|^2u_{h,\tau},\psi) \d t = 0
\end{split}
\end{align}
for all $\psi \in S_h$, all $i = 1,\dots,q$ and $n = 0,\dots,N-1$. Here the coefficients $m_{ij}$ are given by
\begin{align} \label{mcoeff}
	m_{ij} = \int_0^1 \hat{\ell}'_{j}(s) \, \ell_{i}(s) \d s \quad \mbox{for } 1\le i \le q,\, 0 \le j \le q.
\end{align}
Note that $u_{h,\tau}^{n,0} = u_{h,\tau}(t_n) = u_{h,\tau}^{n+}$ is given by the previous time step via
\begin{align*}
	u_{h,\tau}^{n,0} = \sum_{j = 0}^q \hat{\ell}_{j}(1) \, u^{n-1,j}_{h,\tau}, \quad n = 1,\dots,N-1, \quad \mbox{and} \quad u_{h,\tau}^{0,0} = P_h u_0.
\end{align*}
The remaining integral of the nonlinear term in \eqref{cGscheme2} is of degree $4q-1$ in time and therefore the $q$-stage Gauss-Legendre quadrature is not exact for this term. However, we can integrate it using a $2q$-stage Gauss-Legendre quadrature rule. For that purpose, let $\tilde{s}^\GL_j$, $j = 1, \dots , 2q$ be the Gaussian quadrature nodes on $[0,1]$ and $\tilde{w}^\GL_j$, $j = 1,\dots,2q$ the weights of the $2q$-stage Gauss-Legendre quadrature rule. Applying the quadrature rule to the integral of the nonlinear term in \eqref{cGscheme2}, and writing compactly $g(u) = |u|^2u$, then leads to
\begin{align} \label{timestepping}
\begin{split}
	\sum_{j = 0}^q m_{ij} \, (u^{n,j}_{h,\tau}, \psi) - \i \tau_n w^\GL_i  \, (u^{n,i}_{h,\tau}, \psi )_\H - \i \beta \tau_n \sum_{\nu = 1}^{2q} \tilde{w}^{\GL}_\nu \, \ell_{i}(\tilde{s}^\GL_\nu) \Big( g \Big( \sum_{j = 0}^q \hat{\ell}_j(\tilde{s}^\GL_\nu) u^{n,j}_{h,\tau} \Big), \psi \Big) = 0
\end{split}
\end{align}
for all $\psi \in S_h$, $i = 1,\dots,q$ and $n = 0,\dots,N-1$. Recall at this point that we use the inner product $(u,v) = \int_\D \overline{u} \,v \d\x$ which is conjugate-linear in the first argument. \\

Next, we describe a way to realize the time stepping procedure $u_{h,\tau}(t_n) \rightarrow u_{h,\tau}(t_{n+1})$ in view of implementing the cG($q$)-method. So we assume that $u_{h,\tau}(t_n) = u^{n,0}_{h,\tau} \in S_h$ is given from the previous time step. What we need to compute are $u^{n,j}_{h,\tau} \in S_h$, $j = 1,\dots,q$ satisfying \eqref{timestepping} so that $u_{h,\tau}(t_{n+1}) \in S_h$ is obtained from
\begin{align*}
	u_{h,\tau}(t_{n+1}) = \sum_{j = 0}^q \hat{\ell}_{j}(1) u^{n,j}_{h,\tau}.
\end{align*}
The system \eqref{timestepping} is a coupled system for the unknowns $u^{n,j}_{h,\tau}$. So in order to decrease the computational effort, one can decouple the system as it was already proposed in \cite{Makridakis99}. For that purpose, we set $\mathcal{M} = (m_{ij})_{i,j = 1,\dots q}$ with $m_{ij}$ from \eqref{mcoeff} and further set $W := \diag (w_1^\GL,\dots,w_q^\GL)$. Then we note that $A = \mathcal{M}^{-1}W$ is the coefficient matrix of the $q$-stage Gauss-Legendre Implicit Runge--Kutta method and is therefore diagonalizable (cf. \cite{Dekker84} and \cite{Makridakis99}). Hence there is $\Sigma \in \R^{q \times q}$ such that $\Sigma A \Sigma^{-1} = \Gamma = \diag (\gamma_1,\dots,\gamma_q)$. Now we introduce
\begin{align*}
	U^{n,i} = \sum_{j = 1}^q \Sigma_{ij} \, u^{n,j}_{h,\tau}, \quad i = 1,\dots,q.
\end{align*}
Then $U^{n,j}$, $j = 1,\dots,q$ solves
\begin{align} \label{timesteppingtransform}
	(U^{n,i}, \psi) - \i \, \tau_n \gamma_i (U^{n,i}, \psi)_\H = a_i(u^{n,0}_{h,\tau}, \psi) + \i \, \beta \tau_n \sum_{\nu = 1}^q b_{i\nu} \Big( g \big( c_{0\nu} u^{n,0}_{h,\tau} + \sum_{j = 1}^q c_{j\nu} U^{n,j} \big), \psi \Big)
\end{align}
for all $i = 1,\dots,q$ and $\psi \in S_h$. Here the coefficients $a_i, b_{i\nu}, c_{0\nu}, c_{i\nu}$ for $i = 1,\dots,q$ and $\nu = 1,\dots, 2q$ are given by
\begin{align*}
	a_i = \sum_{j = 1}^q \Sigma_{ij}, \quad b_{i\nu} = \tilde{w}_\nu^\GL \sum_{j = 1}^q (\Sigma \mathcal{M})^{-1}_{ij} \ell_j(\tilde{s}^\GL_\nu), \quad c_{0\nu} = \hat{\ell}_0(\tilde{s}_\nu^\GL), \quad c_{i\nu} = \sum_{j = 1}^q \hat{\ell}_j(\tilde{s}^\GL_\nu) \Sigma^{-1}_{ji}.
\end{align*}
So in each time step, one has to solve the now decoupled system \eqref{timesteppingtransform} for $U^{n,i}_{h,\tau}$, $i = 1,\dots,q$. In order to solve the nonlinear system \eqref{timesteppingtransform} we propose a fixed point iteration that we explain briefly in the following. Let us define $\H_h: S_h \rightarrow S_h$ via $(\H_h u , v) = (u,v)_\H$. Then we can write \eqref{timesteppingtransform} as
\begin{align} \label{finaltimestepping}
	(I + \i \tau_n \gamma_i \H_h) U^{n,i} = a_i u^{n,0}_{h,\tau} - \i \beta \sum_{\nu = 1}^q b_{i\nu} g \big( c_{0\nu} u^{n,0}_{h,\tau} + \sum_{j = 1}^q c_{j\nu} U^{n,j} \big), \quad i =1,\dots,q,
\end{align}
where $I$ denotes the identity.
Now the explicit fixed point iteration reads as follows: Set $U^{n,i}_{0} = U^{n-1,i}$ for $i = 1,\dots,q$ and iterate for $k = 0,\dots$ the system of linear equations
\begin{align} \label{fixedpointiteration}
	(I + \i \tau_n \gamma_i \H_h) U^{n,i}_{k+1} = a_i u^{n,0}_{h,\tau} - \i \beta \sum_{\nu = 1}^q b_{i\nu} g \big( c_{0\nu} u^{n,0}_{h,\tau} + \sum_{j = 1}^q c_{j\nu} U^{n,j}_k \big), \quad i =1,\dots,q
\end{align}
until $\| U^{n,i}_{k+1} - U^{n,i}_{k} \| < \varepsilon_{\mathrm{FP}}$ for some tolerance $\varepsilon_{\mathrm{FP}} > 0$ and some suitable norm $\| \cdot \|$ on $S_h$. Of course the stopping criterion can also be chosen w.r.t. the residuum of \eqref{finaltimestepping}. However, we note that after formulating \eqref{finaltimestepping} in the finite element basis the resulting matrix associated with the operator $(I + \i \tau_n \gamma_i \H_h)$ does not change in the time stepping if $\tau_n=\tau$ is selected as a uniform constant. Hence, it can be $\mathbf{LU}$-decomposed in a pre-process which makes solving the linear equations in \eqref{fixedpointiteration} during the time stepping more efficient.

\sect{Numerical scheme with truncated nonlinearity}
\label{sec:truncetedproblem}

In this section we introduce a truncated scheme which is used to derive the stated a priori estimates for the numerical solution from \eqref{cGscheme}. The truncation technique is a classical approach in the context of nonlinear Schr\"odinger equations that allows to get an a priori control over the growth of the nonlinear term (cf. \cite{ADK91,HenningMalqvist17,HeP17,HeW22,Makridakis98,Makridakis99,Zou01}). 
For that, we start from the fully-discretized problem \eqref{cGscheme} and replace the nonlinear term $|u|^2u$ by a cutoff function $f(u)$ which is Lipschitz continuous w.r.t the $L^2$- and $H^1$-norm. After that, we prove the desired a priori estimates for the truncated problem. Our main result then guarantees that the solution of the truncated problem coincides with the one of the original numerical scheme \eqref{cGscheme}, which concludes the argument. In the next lemma we introduce the cutoff function as suggested in \cite{Makridakis98} and we present its properties. Slightly 
sharper versions can be found in the literature; see e.g. \cite[Proof of Lemma 5.2]{HenningMalqvist17}.

\begin{lemma}[{\cite[Lemma 4.1 and its proof]{Makridakis98}}] \label{cutoff}
Let assumptions \ref{A1}, \ref{A7} and \ref{A8} be fulfilled and $M > 0$ given by \eqref{M}. Then there is a $C^2$ function $\gamma: [0,\infty) \rightarrow [0,\infty)$ with $\gamma(0) = 0$ and a constant $C = C(M) > 0$ such that $f(z) = \gamma(|z|^2)z$ satisfies
\begin{enumerate}[label={\roman*).}]
\item[\mbox{\normalfont i)}] $f(z) = |z|^2 z$ for all $|z| < M$,
\item[\mbox{\normalfont ii)}] $|f(z)| \le C|z|$ for all $z \in \C$,
\item[\mbox{\normalfont iii)}] $|f(z) - f(w)| \le C|z - w|$ for all $z,w \in \C$
\item[\mbox{\normalfont iv)}] $\| \nabla (f(v) - f(w)) \|_{L^2(\D)} \le C \| \nabla( v - w )\|_{L^2(\D)}$ for all $\| v \|_{W^{1,\infty}(\D)} < \| u \|_{L^\infty(I,W^{1,\infty}(\D))} + 1$ and $w \in H^1(\D)$.
\end{enumerate} 
\end{lemma} 

Note that assumptions \ref{A7} and \ref{A8} in Lemma \ref{cutoff} are only needed in order to define the cutoff region determined by $M$ from \eqref{M}. Using the cutoff function $f$ we can now state the truncated problem as follows:

\begin{definition}[Truncated fully-discrete cG($q$)-scheme for GPE]
Under the assumptions \ref{A1}-\ref{A8} and for given $u_0 \in H^1_0(\D)$. The truncated fully-discrete approximation is given by a function $\ut_{h,\tau} \in V_q$ such that for all $n = 0,\dots, N-1$
\begin{align} \label{auxproblem}
\begin{split}
	\int_{I_n} (\i \partial_t \ut_{h,\tau}, v) - (\ut_{h,\tau}, v)_{\H} - \beta (f(\ut_{h,\tau}),v) \d t & = 0 \quad \forall v \in V_{q-1}, \\
	\ut_{h,\tau}^{n+} & = \ut_{h,\tau}(t_{n}),
\end{split}
\end{align}
where $\ut_{h,\tau}^{n+} = \lim_{t \searrow t_n} \ut_{h,\tau}(t)$, $\ut_{h,\tau}^{0+} = P_h u_0$ and $f$ is the cutoff function from \eqref{cutoff} with $M$ given by \eqref{M}.
\end{definition}

Using the same arguments as in the proof of Proposition \ref{energyconservation} we have that every solution $\ut_{h,\tau}$ of \eqref{auxproblem} preserves the truncated energy
\begin{align} \label{modified_energy}
	\tilde{E}(u) := \frac{1}{2} \int_{\D} |\nabla u|^2 - \Omega \, \overline{u} \, \L_z u + V |u|^2 + \beta \, \Gamma(|u|^2) \d\x, \quad \Gamma(s) := \int_0^s \gamma(t) \d t
\end{align}
at the discrete time instances $t_n$, i.e., $\tilde{E}( u_{h,\tau}(t_n)) = \tilde{E}(P_h u_0)$ for all $n = 1,\dots,N$. The existence of a solution of \eqref{auxproblem} was shown in \cite{Makridakis99} for the case $V = 0$ and $\Omega = 0$ and is similar to the case considered in \cite{Makridakis98}. The proof is based on Browder's fixed-point theorem (cf. \cite[Lemma 4]{Browder1965} or \cite[Lemma 3.1]{ADK91}) and the following stability result of the numerical time stepping scheme.

\begin{lemma}[{\cite[Lemma 2.1]{Makridakis99}}] \label{stability}
Let $\mathcal{M} = (m_{ij})_{i,j = 1,\dots,q}$ be given by \eqref{mcoeff}. Further, let $s^\GL_i$, $i = 1,\dots, q$ be the nodes of the $q$-stage Gauss-Legendre quadrature rule on $[0,1]$ and let 
\begin{align}
\label{definition_tilde_m}
\mathbf{M} := D^{-1/2} \mathcal{M} \, D^{1/2} \quad \mbox{with  } D := \mathrm{diag}(s^\GL_1,\dots,s^\GL_q).
\end{align}
Then there is $\alpha > 0$ such that
\begin{align*}
	\Re \big( \x \cdot \mathbf{M}\x \big) \ge \alpha |\x|^2 \quad \forall \x \in \C^q.
\end{align*}
\end{lemma}

The existence of a solution of the truncated problem \eqref{auxproblem} in the general case with potential $V$ and angular velocity $\Omega$ is now a straightforward generalization of the result from \cite{Makridakis99} and \cite{Makridakis98}. Therefore we omit the proof here and only state the result in the following lemma.

\begin{lemma} \label{existencefull}
Under the assumptions \ref{A1}-\ref{A8} there are $\tau_0,h_0 > 0$ such that for every $0 < \tau < \tau_0$ and $0 < h < h_0$ the problem \eqref{auxproblem} has at least one solution $\ut_{h,\tau} \in V_q$.
\end{lemma}

\sect{Error analysis of the semi-discrete scheme}
\label{sec:errorsemi}

In this section we introduce the semi-discrete version (discrete in time and continuous in space) of the numerical scheme with truncated nonlinearity as formulated in \eqref{auxproblem}. This semi-discrete auxiliary problem is crucial for ultimately avoiding mesh coupling conditions in the error analysis. We prove that the auxiliary problem is well-posed and we derive error estimates for the corresponding solution. In particular, we show $L^\infty(L^\infty)$-estimates for the semi-discrete approximation of the truncated problem such that the attained approximation coincides with the corresponding semi-discrete approximation of the original problem \eqref{weakGPE} (i.e. without truncation).\\[0.5em]
We start with defining the semi-discrete spaces
\begin{align*}
	\W_q & := \Big\{ v: \D \times I \rightarrow \C\,: \, v_{|\D \times I_n}(x,t) = \sum_{j = 0}^q t^j \chi_j(x), \, \chi_j \in H^1_0(\D) \Big\}, \\
	 \W_q^n & := \{ v_{|\D \times I_n}: v \in \W_q \},
\end{align*}
and then introduce the semi-discrete truncated problem as follows:

\begin{definition}[Truncated semi-discrete cG($q$)-scheme for GPE]
Assume \ref{A1}-\ref{A4}, \ref{A7}, \ref{A8} and let $u_0 \in H^1_0(\D)$. The truncated semi-discrete approximation is given by a function $\ut_{\tau} \in \W_q$ such that for all $n = 0,\dots, N-1$
\begin{align} \label{semiproblem}
\begin{split}
	\int_{I_n} (\i \partial_t \ut_{\tau}, v) - (\ut_{\tau}, v)_{\H} - \beta (f(\ut_{\tau}),v) \d t & = 0 \quad \forall v \in \W_{q-1}, \\
	\ut_{\tau}^{n+} & = \ut_{\tau}(t_{n}),
\end{split}
\end{align}
where $\ut_{\tau}^{n+} = \lim_{t \searrow t_n} \ut_{\tau}(t)$ and $\ut_{\tau}^{0+} = u_0$. \\
\end{definition}

\subsection{Existence of the truncated semi-discrete approximation}

We show that the truncated semi-discrete problem \eqref{semiproblem} has at least one solution.

\begin{lemma} \label{existencesemi}
Under the assumptions \ref{A1}-\ref{A4}, \ref{A7} and \ref{A8} there is $\tau_0 > 0$ such that for every $0 < \tau < \tau_0$ the problem \eqref{semiproblem} has at least one solution $\ut_{\tau} \in \W_q$. Moreover, $\ut_\tau \in C(I,H^1_0(\D) \cap H^2(\D))$.
\end{lemma}

\begin{proof}
We choose a finite dimensional subspace $S_h \subset H^1_0(\D)$ such that assumptions \ref{A5}-\ref{A7} are satisfied by this space. For instance, this can be achieved by taking $S_h$ as the space of linear Lagrange finite elements on a (quasi) uniform mesh with mesh size $h > 0$. Then Lemma \ref{existencefull} guarantees the existence of $\ut_{h,\tau} \in V_q$ for every $h > 0$ satisfying \eqref{auxproblem}. Next we pass to the limit $h \rightarrow 0$ and show that $\ut_{h,\tau} \in V_{q}$ converges to some $\ut_\tau \in \W_q$ satisfying \eqref{semiproblem}. Now $\ut_{h,\tau}$ preserves the truncated energy $\tilde{E}$ from \eqref{modified_energy} and by \ref{A7} we have $\| P_h u_0 \|_{L^{\infty}(\D)} \le C_\infty | u_0 |_{H^2(\D)} + \| u_0 \|_{L^\infty(\D)} < M$. So since $\beta \ge 0$ by \ref{A2} we can estimate
\begin{align*}
	c \| \nabla \ut_{h,\tau}(t_n) \|_{L^2(\D)}^2 \le \tilde{E}(\ut_{h,\tau}(t_n)) = \tilde{E}( P_h u_0) = E(P_h u_0) & \le C \| P_h u_0 \|_{H^1(\D)}^2 + C \| P_h u_0 \|_{H^1(\D)}^4 \\
	& \le C \| u_0 \|_{H^1(\D)}^2 + C \| u_0 \|_{H^1(\D)}^4
\end{align*}
for some generic constants $c,C > 0$ independent on $h$. Here we used the Sobolev embedding $H^1_0(\D) \subset L^4(\D)$ and \ref{A5}. Hence $\ut_{h,\tau}(t_n)$ is uniformly bounded in $H_0^1(\D)$ for all $h > 0$ and $n = 0,\dots,N-1$. Now $\ut_{h,\tau}$ can be expressed uniquely on $I_n$ as
\begin{align} \label{auxproof}
	\ut_{h,\tau}(t) = \sum_{j = 0}^q \hat{\ell}_{n,j}(t) \, \ut_{h,\tau}^{n,j}, \quad \mbox{for } t \in I_n,
\end{align}
for some $\ut_{h,\tau}^{n,j} \in S_h$, $j = 1,\dots,q$ and with $\ut_{h,\tau}^{n,0} = \ut_{h,\tau}(t_n) \in S_h$. Introducing $\tilde{\mathbf{u}}_{h,\tau}^{n,j} = (s^\GL_j)^{-\frac{1}{2}} \ut_{h,\tau}^{n,j}$, $j = 1,\dots,q$ equation \eqref{auxproblem} can be equivalently rewritten as
\begin{align} \label{auxproblem2}
	\sum_{j = 1}^q \mathbf{m}_{ij} (\tilde{\mathbf{u}}_{h,\tau}^{n,j},\psi) - \i \tau_n w^\GL_i ( \tilde{\mathbf{u}}_{h,\tau}^{n,i}, \psi )_{\H} - \i \beta \int_{I_n} (s^\GL_i)^{-\frac{1}{2}} \ell_{n,i} (f(\ut_{h,\tau}),\psi) \d t + (s^\GL_i)^{-\frac{1}{2}} m_{i0} (\ut_{h,\tau}^{n,0}, \psi) = 0
\end{align}
for all $i = 1,\dots,q$ and $\psi \in S_h$ and with $\mathbf{m}_{ij} = (\mathbf{M})_{ij}$ from \eqref{definition_tilde_m}. Now we test in \eqref{auxproblem2} with $\psi = \tilde{\mathbf{u}}_{h,\tau}^{n,i}$, sum over $i = 1,\dots,q$, take the real part and use Lemma \ref{stability} to obtain
\begin{align*}
	\alpha \sum_{i = 1}^q \| \tilde{\mathbf{u}}_{h,\tau}^{n,i} \|_{L^2(\D)}^2 \le C\tau_n \sum_{i = 1}^q \| \tilde{\mathbf{u}}_{h,\tau}^{n,i} \|_{L^2(\D)}^2 + C \| \ut_{h,\tau}(t_n) \|_{L^2(\D)} \Big( \sum_{i = 1}^q \| \tilde{\mathbf{u}}_{h,\tau}^{n,i} \|_{L^2(\D)}^2 \Big)^{1/2}.
\end{align*}
So for $\tau$ sufficiently small this shows uniform bounds of $\tilde{\mathbf{u}}_{h,\tau}^{n,i}$ in $L^2(\D)$. On the other hand, when taking the imaginary part we obtain
\begin{align*}
	\| \tilde{\mathbf{u}}_{h,\tau}^{n,i} \|_{\H}^2 & \le \Big| \sum_{j = 1}^q \mathbf{m}_{ij} ( \tilde{\mathbf{u}}_{h,\tau}^{n,j}, \tilde{\mathbf{u}}_{h,\tau}^{n,i}) - \i\beta \int_{I_n} (s^\GL_i)^{-\frac{1}{2}} \ell_{n,i} (f(\tilde{u}_{h,\tau}),\tilde{\mathbf{u}}_{h,\tau}^{n,i}) \d t + (s^\GL_i)^{-\frac{1}{2}} m_{i0} (\ut_{h,\tau}(t_n), \tilde{\mathbf{u}}_{h,\tau}^{n,i}) \Big| \\
	& \le C \sum_{j = 1}^q \| \tilde{\mathbf{u}}_{h,\tau}^{n,j} \|_{L^2(\D)}^2 + C \| \ut_{h,\tau}(t_n) \|_{L^2(\D)} \| \tilde{\mathbf{u}}_{h,\tau}^{n,i} \|_{L^2(\D)}.
\end{align*}
So by Lemma \ref{elliptic} this shows that $\tilde{\mathbf{u}}_{h,\tau}^{n,j}$ is uniformly bounded in $H^1_0(\D)$ for every $n = 0,\dots,N-1$, $j = 1,\dots,q$ and $h > 0$. Next we set $\tilde{\mathbf{u}}_{h,\tau}^{n,0} = \ut_{h,\tau}(t_n)$ which we proved to be uniformly bounded in $H^1_0(\D)$ for all $n = 0,\dots,N-1$ and $h > 0$. Then there is a sequence $h_N \rightarrow 0$, $N \rightarrow \infty$ and weak limits $\tilde{\mathbf{u}}_\tau^{n,j} \in H^1_0(\D)$, $j = 0,\dots,q$, $n = 0,\dots, N-1$ such that
\begin{align*}
	& \tilde{\mathbf{u}}_{h_N,\tau}^{n,j} \rightarrow \tilde{\mathbf{u}}_{\tau}^{n,j} \quad \text{strongly in }L^2(\D) \quad \text{and} \quad \tilde{\mathbf{u}}_{h_N,\tau}^{n,j} \rightarrow \tilde{\mathbf{u}}_{\tau}^{n,j} \quad \text{weakly in }H^1_0(\D)
\end{align*}
for $N \rightarrow \infty$. This implies
\begin{align*}
	(\tilde{\mathbf{u}}_{h_N,\tau}^{n,j}, \psi) \rightarrow (\tilde{\mathbf{u}}_{\tau}^{n,j}, \psi) \quad \text{and} \quad ( \tilde{\mathbf{u}}_{h_N,\tau}^{n,j} , \psi )_\H \rightarrow (\tilde{\mathbf{u}}_{\tau}^{n,j}, \psi )_\H
\end{align*}
for all $\psi \in H^1_0(\D)$. Next we set $\tilde{\mathbf{u}}_{\tau}(t) := \sum_{i = 0}^q \hat{\ell}_{n,i}(t) \tilde{\mathbf{u}}_{\tau}^{n,i}$ and $\tilde{\mathbf{u}}_{h_N,\tau}(t) = \sum_{i = 0}^q \hat{\ell}_{n,i}(t) \tilde{\mathbf{u}}_{h_N,\tau}^{n,i}$ for $t \in I_n$. Then we have
\begin{align*}
	\int_{I_n} \| \tilde{\mathbf{u}}_{\tau}(t) - \tilde{\mathbf{u}}_{h_N,\tau}(t) \|_{L^2(\D)} \d t \le \sum_{i = 0}^q \int_{I_n} |\hat{\ell}_{n,i}(t)| \d t \, \| \tilde{\mathbf{u}}_{\tau}^{n,i} - \tilde{\mathbf{u}}_{h_N,\tau}^{n,i} \|_{L^2(\D)} \rightarrow 0, \quad N \rightarrow \infty.
\end{align*}
This gives with Lemma \ref{cutoff}
\begin{align*}
	\Big \| \int_{I_n} (s^\GL_i)^{-\frac{1}{2}} \ell_{n,i}(t) \big( f(\tilde{\mathbf{u}}_{\tau}(t)) - f(\tilde{\mathbf{u}}_{h_N,\tau}(t)) \big) \d t \Big\|_{L^2(\D)} \le C \int_{I_n} \| \tilde{\mathbf{u}}_{\tau}(t) - \tilde{\mathbf{u}}_{h_N,\tau}(t) \|_{L^2(\D)} \d t \rightarrow 0
\end{align*}
for $N \rightarrow \infty$ and we conclude
\begin{align*}
	\beta \int_{I_n} (s^\GL_i)^{-\frac{1}{2}} \ell_{n,i}(t) (f(\tilde{\mathbf{u}}_{h_N,\tau}(t)),\psi) \d t \rightarrow \beta \int_{I_n} (s^\GL_i)^{-\frac{1}{2}} \ell_{n,i}(t) (f(\tilde{\mathbf{u}}_{\tau}(t)),\psi) \d t
\end{align*}
for all $\psi \in H^1_0(\D)$. Collecting all the convergence results and passing to the limit yields that $\ut_{\tau}$ defined by
\begin{align*}
	\ut_\tau(t) = \sum_{i = 0}^q \hat\ell_{n,i}(t) (s^\GL_i)^\frac{1}{2} \tilde{\mathbf{u}}_\tau^{n,i}, \quad t \in I_n
\end{align*}
solves \eqref{semiproblem} on $I_n$. In particular, the convergence $\ut_{h,\tau} \rightarrow \ut_\tau$ is uniformly in $\overline{I}$ and hence $\ut_\tau \in \W_q$. In view of \eqref{auxproblem2} the regularity of $\ut_{\tau}$ is a standard $H^2$-regularity argument for elliptic problems on convex polyhedral domains; see e.g. \cite{GilbargTrudinger01}. This completes the proof.
\end{proof}

\subsection{Error estimates of the truncated semi-discrete approximation}

In this section we derive a priori error estimates of the truncated semi-discrete approximation from \eqref{semiproblem} in the $L^\infty(I,H^1(\D))$- and $L^\infty(I,H^2(\D))$-norm. The latter one then allows for an uniform estimate in $L^\infty(I \times \D)$ of the semi-discrete approximation. This is an essential interim result in order to prove our main result. Before starting the error analysis we note the following elementary lemma.

\begin{lemma} \label{normequiv}
There are constants $C_1,C_2 > 0$ such that for all $v \in \W_q^n$ with $v = \sum_{j=0}^q \hat{\ell}_{n,j} v_j$, $v_j \in H^1_0(\D)$ it holds
\begin{align*}
	C_1 \sum_{j=0}^q \tau_n \| v_j \|_{L^2(\D)}^2 \le \| v \|_{L^2(I_n \times \D)}^2 \le C_2 \sum_{j=0}^q \tau_n \| v_j \|_{L^2(\D)}^2.
\end{align*}
\end{lemma}

\begin{proof}
The estimate from below follows by the inverse estimate $\| \hat{\ell}_{n,i} \|_{L^\infty(I_n)} \le C\tau_n^{-1/2} \| \hat{\ell}_{n,i} \|_{L^2(I_n)}$, see \cite[Lemma 4.5.3]{BrennerScott08}). The estimate from above is obtained directly by
\begin{align*}
	\| v \|_{L^2(I_n \times \D)}^2 \le \sum_{j = 0}^q \int_{I_n} |\hat{\ell}_{n,j}(t)|^2 \d t \| v_j \|_{L^2(\D)}^2 \le C_2 \sum_{j=0}^q \tau_n \| v_j \|_{L^2(\D)}^2.
\end{align*}
\end{proof}

Now the idea of the further error analysis in this section is to compare the truncated semi-discrete approximation $\ut_\tau$ from \eqref{semiproblem} with the interpolation of the solution $u$ at the $q+1$-Gauss-Lobatto quadrature nodes on $\overline{I_n}$. The quadrature nodes are denoted by $t_n = t^\Lo_{n,0} < t^\Lo_{n,1} < \cdots < t^\Lo_{n,q} = t_{n+1}$ and are chosen such that the Gauss-Lobatto quadrature rule satisfies
\begin{align*}
	\int_{I_n} p(s) \d s = \sum_{j = 0}^{q} w^\Lo_{n,j} \,\, p(t^\Lo_{n,j}) \quad \forall p \in \mathcal{P}^{2q-1}(I_n)
\end{align*}
with suitable weights $w^\Lo_{n,j}$, $j = 0,\dots, q$. Then we define the interpolation operator w.r.t. the nodes $t^\Lo_{n,0}, \dots,t^\Lo_{n,q}$ on $\overline{I}$ by
\begin{align} \label{GaussLobattoInterpolation}
\begin{split}
	& \I^\Lo_\tau: C(\overline{I},L^2(\D)) \rightarrow C(\overline{I},L^2(\D)) \quad \text{s.t.} \quad (\I^\Lo_\tau v)_{|I_n} \in \mathcal{P}^{q}(I_n) \quad \text{and} \quad (\I^\Lo_\tau v)(t^\Lo_{n,j}) = v(t^\Lo_{n,j})
\end{split}
\end{align}
for all $n = 0,\dots,N-1$ and $j = 0,\dots,q$. \\
Using the Gauss-Lobatto interpolation $\I^\Lo_\tau u$ we now split the error into
\begin{align} \label{errorsplit}
	u - \ut_\tau = u- \I^\Lo_\tau u + e_\tau, \quad e_\tau := \I^\Lo_\tau u - \ut_\tau.
\end{align}
Then standard Lagrange interpolation estimates (see e.g. \cite[Theorem 4.4.4]{BrennerScott08}) imply for $p = 2$ or $p = \infty$ the estimate
\begin{align} \label{LagrangeEstimate}
	\| v - \I^\Lo_\tau v \|_{L^p(I_n,L^2(\D))} \le C \tau_n^{q+1} \| \partial_t^{q+1} v \|_{L^p(I_n,L^2(\D))}, \quad \forall v \in W^{q+1,p}(I_n,L^2(\D))
\end{align}
for some constant $C > 0$. Hence the first term in \eqref{errorsplit} is easily estimated in appropriate norms due to the regularity assumption \ref{A8} on $u$. So we now focus on $e_\tau$ and first introduce the notation
\begin{align*}
	e_\tau^{n,j} := e_\tau(t^\GL_{n,j}), \quad e_\tau^n := e_\tau(t_n), \quad \text{for} \quad n = 0,\dots,N-1, \quad j = 0,\dots,q.
\end{align*}
Then a short computation shows that $e_\tau$ satisfies the error equation
\begin{align} \label{erroreq}
\begin{split}
	\sum_{j=0}^q m_{ij} (e_\tau^{n,j}, \psi) - \i \tau_n w^\GL_i & ( e_\tau^{n,i}, \psi )_\H - \i \beta \int_{I_n} \ell_{n,i}(f(\I^\Lo_\tau u) - f(\ut_\tau),\psi) \d t \\
	&  = (A^{n,i} - \i B^{n,i}, \psi) - \i \beta \int_{I_n} \ell_{n,i} (f(\I^\Lo_\tau u) - f(u), \psi)\d t
\end{split}
\end{align}
for all $\psi \in H^1_0(\D)$, $i = 1,\dots,q$ and with the quadrature error terms
\begin{align*}
	A^{n,i} & := \sum_{j=0}^q w^\Lo_{n,j} \ell_{n,i}'(t^\Lo_{n,j}) u(t^\Lo_{n,j}) -  \int_{I_n} \ell_{n,i}' u \d t, \\
	B^{n,i} & := \sum_{j=0}^q w^\Lo_{n,j} \ell_{n,i}(t^\Lo_{n,j}) \H u(t^\Lo_{n,j}) - \int_{I_n} \ell_{n,i} \H u \d t.
\end{align*}
Note that by the regularity assumption \ref{A8}, the solution $u$ satisfies $\H u \in C(I,H^1_0(\D))$ so that $A^{n,i}$ and $B^{n,i}$ are well-defined and belong to $H^1_0(\D)$. Futhermore, since $e_\tau^{n,i} \in H^2(\D)$ we can rewrite the critical term $(e_{\tau}^{n,i},\psi)_{\H}$ in \eqref{erroreq} as $(\H e_\tau^{n,i}, \psi)$ which is even well-defined for $\psi \in L^2(\D)$. From there, we conclude that the error equation \eqref{erroreq} even holds for all $\psi \in L^2(\D)$. \\
The rest of this section is structured as follows: First we derive uniform estimates for $e_\tau$ w.r.t. the $H^1(\D)$-norm which will lead to corresponding $H^1(\D)$-estimates for the error $u - \ut_\tau$. The results are then used to conclude uniform estimates in the $H^2(\D)$-norm. Sobolev embedding will then imply estimates w.r.t. $L^\infty(I \times \D)$ for the error $u - \ut_\tau$ and therefore of the truncated semi-discrete approximation $\ut_\tau$ from \eqref{semiproblem} as well.

\begin{remark}
In \cite{Makridakis99} the error of the time discretization is handled by an interplay of the Gauss-Lobatto interpolation and the Gauss-Legendre interpolation. The first one is used to handle the consistency whereas the stability of the scheme is shown by the important observation that the Gauss-Legendre interpolation coincides with the $L^2(I_n)$-projection of $\mathcal{P}^{q}(I_n)$ onto $\mathcal{P}^{q-1}(I_n)$ (cf. \eqref{interpolation-projection}). We use the same idea for the $L^2(H^1(\D))$-estimates of the truncated semi-discrete approximation and to conclude the $L^\infty(I \times \D)$-estimates of the fully-discrete approximation in section \ref{section:proof-main-result}. However, for the important $L^\infty(H^2(\D))$-estimates of the truncated semi-discrete approximation, the stability is handled in a different way, starting from a more general form of \eqref{erroreq} (cf. \eqref{erroreq0}) and not exploiting that the interpolation coincides with the projection.
\end{remark}

\subsubsection{$L^\infty(I,H^1(\D))$-estimates} \label{secH1semi}

We start by estimating the quadrature error terms $A^{n,i}$ and $B^{n,i}$.

\begin{lemma} \label{ABest}
Under the assumptions \ref{A1}-\ref{A4} and \ref{A8} there are constants $C,\tau_0 > 0$ such that for all $0 < \tau_n < \tau_0$, $i = 1,\dots,q$ and $n = 0,\dots,N-1$ there hold
\begin{align*}
	\| \nabla A^{n,i} \|_{L^2(\D)} & \le C \tau_n^{q + \frac{3}{2}} \| \partial_t^{q+2} \nabla u \|_{L^2(I_n \times \D)}, \\
	\| \nabla B^{n,i} \|_{L^2(\D)} & \le C \tau_n^{q + \frac{3}{2}} \| \partial_t^{q+1}\nabla \H u \|_{L^2(I_n \times \D)}.
\end{align*}
\end{lemma}

\begin{proof}
We start with the second estimate of $\nabla B^{n,i}$ and recall the interpolation operator $\I^\Lo_\tau$ from \eqref{GaussLobattoInterpolation}. Since $\ell_{n,i} \, \I^\Lo_\tau \nabla  \H u$ is a polynomial of degree $2q-1$ on $I_n$, we obtain by the exactness of the Gauss-Lobatto quadrature 
\begin{align*}
	\nabla B^{n,i} = \int_{I_n} \ell_{n,i} (\I^\Lo_\tau - I )\nabla \H u \d t .
\end{align*}
So using the interpolation estimate \eqref{LagrangeEstimate} we obtain
\begin{align*}
	\| \nabla B^{n,i} \|_{L^2(I_n \times \D)} \le\Big( \int_{I_n} |\ell_{n,i}|^2 \d t \Big)^{\frac{1}{2}} \| (I - \I^\Lo_\tau)\nabla \H u \|_{L^2(I_n \times \D)} \le C \tau_n^{q + \frac{3}{2}} \| \partial^{q+1}_t \nabla \H u \|_{L^2(I_n \times \D)}.
\end{align*}
To estimate $\nabla A^{n,i}$ with optimal order, we introduce $\I^{\Lop}_\tau$ to be the Lagrange interpolation to the $(q+1)$ Gauss-Lobatto points $t^\Lo_{n,j}$ plus any other point in $I_n$ different from $t^\Lo_{n,j}$. A corresponding standard error estimate (cf. \cite{BrennerScott08}) guarantees
\begin{align*}
	\| (I - \I^{\Lop}_\tau) v \|_{L^2(I_n \times \D)} \le C \tau_n^{q + 2} \| \partial_t^{q+2} v \|_{L^2(I_n \times \D)} \quad \forall v \in H^{q+2}(I,L^2(\D)).
\end{align*}
Since $\ell'_{n,i} \I^{\Lop}_\tau u$ is a polynomial of degree $2q-1$ on $I_n$ we obtain
\begin{align*}
	\nabla A^{n,i} = \int_{I_n} \ell'_{n,i} (\I^{\Lop}_\tau - I) \nabla u \d t.
\end{align*}
Therefore,
\begin{eqnarray*}
	\| \nabla A^{n,i} \|_{L^2(I_n \times \D)} & \le &\Big( \int_{I_n} |\ell'_{n,i}|^2 \d t \Big)^{\frac{1}{2}} \| (I - \I^{\Lop}_\tau) \nabla u \|_{L^2(I_n \times \D)} \\
	& \le& C \, \tau_n^{-\frac{1}{2}} \| (I - \I^{\Lop}_\tau) \nabla u \|_{L^2(I_n \times \D)}  \hspace{5pt}
	\le \hspace{5pt} C \, \tau_n^{q + \frac{3}{2}} \|  \partial_t^{q+2} \nabla u \|_{L^2(I_n \times \D)}.
\end{eqnarray*}
\end{proof}

The basic strategy now is to test the error equation \eqref{erroreq} with $\psi = \H e_\tau^{n,i} \in L^2(D)$ (which is admissible), summing over $i = 1,\dots,q$ and taking real parts. Recalling the energy norm $\| \cdot \|_{\H}$ from Lemma \ref{elliptic}, we have the following local error estimate for $e_\tau$.

\begin{lemma} \label{consistencyH1}
Under the assumptions \ref{A1}-\ref{A4}, \ref{A7} and \ref{A8} there are constants $C,\tau_0 > 0$ such that for all $0 < \tau_n < \tau_0$ and $n = 0,\dots,N-1$ it holds
\begin{align*}
	\| \nabla e_\tau \|_{L^2(I_n \times \D)}^2 \le C \, \tau_n \, \| e_\tau^n \|_{\H}^2 + C_n(u) \, \tau_n^{2q+4}
\end{align*}
where
\begin{align*}
	C_n(u) = C \Big( \| \partial_t^{q+1} \nabla u \|_{L^2(I_n \times \D)}^2 + \| \partial_t^{q+2} \nabla u \|_{L^2(I_n \times \D)}^2 + \| \partial_t^{q+1} \nabla \H u \|_{L^2(I_n \times \D)}^2 \Big).
\end{align*}
\end{lemma}

\begin{proof}
Similarly, as in the proof of Lemma \ref{existencefull} we introduce $\mathbf{e}_\tau^{n,j} = (s^\GL_j)^{-\frac{1}{2}} e_\tau^{n,j}$, $j = 1,\dots,q$. Then testing in \eqref{erroreq} with $\psi = (s^\GL_i)^{-\frac{1}{2}} \H \mathbf{e}_\tau^{n,i} \in L^2(\D)$, summing over $i = 1,\dots,q$ and taking real parts yields
\begin{align} \label{erroreqL2}
\begin{split}
	\Re \sum_{i,j=1}^q \mathbf{m}_{ij} (\mathbf{e}_\tau^{n,j}, \mathbf{e}_\tau^{n,i})_\H & =  \Im \beta \sum_{i = 0}^q \int_{I_n} (s^\GL_i)^{-\frac{1}{2}} \ell_{n,i}(f(\ut_\tau) - f(\I^\Lo_\tau u),\mathbf{e}_\tau^{n,i})_\H \d t \\
	& \quad + \Im \beta \sum_{i = 0}^q \int_{I_n} (s^\GL_i)^{-\frac{1}{2}} \ell_{n,i} (f(\I^\Lo_\tau u) - f(u), \mathbf{e}_\tau^{n,i})_\H \d t \\
	& \quad + \Re \sum_{i = 0}^q (s^\GL_i)^{-\frac{1}{2}} (A^{n,i} - \i B^{n,i}, \mathbf{e}_\tau^{n,i})_\H  \\
	& \quad - \Re \sum_{i = 0}^q (s^\GL_i)^{-\frac{1}{2}} m_{i0} \ell_{n,i}(t_n) (e_\tau^n,  \mathbf{e}_\tau^{n,i})_\H .
\end{split}
\end{align} 
Next we estimate the terms on the right hand side. By standard Lagrange interpolation estimates (cf. \cite[Theorem 4.4.4]{BrennerScott08} and \eqref{LagrangeEstimate}) we have for all $t \in I$
\begin{align} \label{W1infty}
\begin{split}
	\| \I_\tau^\Lo u(t) \|_{W^{1,\infty}(\D)} & \le \| u(t) \|_{W^{1,\infty}(\D)} + \| u(t) - \I_\tau^\Lo u(t) \|_{W^{1,\infty}(\D)} \\
	& \le \| u(t) \|_{W^{1,\infty}(\D)} + C \tau_n^{q+1} \| \partial_t^{q+1} u(t) \|_{W^{1,\infty}(\D)} < \| u \|_{L^\infty(I,W^{1,\infty}(\D))} + 1
\end{split}
\end{align}
if we choose $\tau_0$ sufficiently small and due to \ref{A8}. Therefore, the Lipschitz estimate from Lemma \ref{cutoff} implies
\begin{eqnarray*}
\lefteqn{ \Big| \beta \sum_{i=1}^q \int_{I_n} (s^\GL_i)^{-\frac{1}{2}} \ell_{n,i} (f(\ut_\tau) - f(\I^\Lo_\tau u), \mathbf{e}_\tau^{n,i})_\H \d t \Big|} \\
	&\le& C \sum_{i=1}^q \int_{I_n} \| f(\ut_\tau) - f(\I^\Lo_\tau u) \|_{H^1(\D)} |\ell_{n,i}| \| \mathbf{e}_\tau^{n,i} \|_{H^1(\D)} \d t \\
	&\le& C \| \nabla e_\tau \|_{L^2(I_n \times \D)} \Big( \sum_{i=1}^q \int_{I_n} |\ell_{n,i}|^2 \d t \| \nabla \mathbf{e}_\tau^{n,i} \|_{L^2(\D)}^2 \Big)^{1/2} \\
		& \le& C \tau_n^{1/2} \| \nabla e_\tau \|_{L^2(I_n \times \D)} \Big( \sum_{i=1}^q \| \nabla \mathbf{e}_\tau^{n,i} \|_{L^2(\D)}^2 \Big)^{1/2}
\end{eqnarray*}
and similarly
\begin{eqnarray*}
	\lefteqn{ \Big| \beta \sum_{i=1}^q \int_{I_n} (s^\GL_i)^{-\frac{1}{2}} \ell_{n,i} (f(u) - f(\I^\Lo_\tau u), \mathbf{e}_\tau^{n,i})_\H \d t \Big| } \\
	& \le& C \tau_n^{1/2} \| \nabla (u - \I^\Lo_\tau u) \|_{L^2(I_n \times \D)} \Big( \sum_{i=1}^q \| \nabla \mathbf{e}_\tau^{n,i} \|_{L^2(\D)}^2 \Big)^{1/2}.
\end{eqnarray*}
Further,
\begin{eqnarray*}
	\lefteqn{\sum_{i=1}^q (s^\GL_i)^{-\frac{1}{2}} (A^{n,i} - \i B^{n,i}, \mathbf{e}_\tau^{n,i})_\H  \hspace{5pt}\le \hspace{5pt}C \sum_{i=1}^q \| A^{n,i} - \i B^{n,i} \|_{H^1(\D)} \| \mathbf{e}_\tau^{n,i} \|_{H^1(\D)} } \\
	& \le& C \Big( \sum_{i=1}^q \| \nabla A^{n,i} \|_{L^2(\D)}^2 + \| \nabla B^{n,i} \|_{L^2(\D)}^2 \Big)^{1/2} \Big( \sum_{i=1}^q \| \nabla \mathbf{e}_{\tau}^{n,i} \|_{L^2(\D)}^2 \Big)^{1/2}
\end{eqnarray*}
and
\begin{align*}
	|\sum_{i = 1}^q (s^\GL_i)^{-\frac{1}{2}} m_{i0} \ell_{n,i}(t_n) (e_\tau^n, \mathbf{e}_\tau^{n,i})_\H| \hspace{5pt}\le\hspace{5pt} C \| \nabla e_\tau^n \|_{L^2(\D)} \Big( \sum_{i=1}^q \| \nabla \mathbf{e}_\tau^{n,i} \|_{L^2(\D)}^2 \Big)^{1/2}.
\end{align*}
Now \eqref{erroreqL2} gives with Lemma \ref{stability} for some $\alpha, \tilde{\alpha} > 0$
\begin{align*}
	\alpha \sum_{i=1}^q \| \nabla e_\tau^{n,i} \|_{L^2(\D)}^2 & \le \tilde{\alpha} \sum_{i=1}^q \| \mathbf{e}_\tau^{n,i} \|_{\H}^2 \le \Re  \sum_{i,j=1}^q \mathbf{m}_{ij} (\mathbf{e}_\tau^{n,j}, \mathbf{e}_\tau^{n,i})_\H \\
	& \le C \Big( \sum_{i=1}^q \| \nabla \mathbf{e}_\tau^{n,i} \|_{L^2(\D)}^2 \Big)^{1/2} \Big\{ \| \nabla e_\tau^n \|_{L^2(\D)} + \tau_n^{1/2} \| \nabla e_\tau \|_{L^2(I_n \times \D)} \\
	& + \tau_n^{1/2} \| \nabla (u - \I^\Lo_\tau u) \|_{L^2(I_n \times \D)} + \Big( \sum_{i=1}^q \| \nabla  A^{n,i} \|_{L^2(\D)}^2 + \| \nabla B^{n,i} \|_{L^2(\D)}^2 \Big)^{1/2} \Big\}.
\end{align*}
With Lemma \ref{normequiv}, Lemma \ref{ABest} and \eqref{LagrangeEstimate} we therefore obtain
\begin{align*}
	\| \nabla e_\tau \|_{L^2(I_n \times \D)}^2 \le C \tau_n \sum_{i=0}^q \| \nabla e_\tau^{n,i} \|_{L^2(\D)}^2 & \le C \tau_n \| \nabla e_\tau^n \|_{\H}^2 + C \tau_n^2 \| \nabla e_\tau \|_{L^2(I_n \times \D)}^2 + C_n(u) \tau_n^{2q + 4}.
\end{align*}
This yields the desired estimate for $\tau_0$ sufficiently small.
\end{proof}

We proceed by showing the global $L^\infty(I,H^1(\D))$-error estimate of the truncated semi-discrete approximation $\ut_\tau$.

\begin{lemma} \label{semiH1estimate}
Under the assumptions \ref{A1}-\ref{A4}, \ref{A7} and \ref{A8} there exists $\tau_0 > 0$ such that we have for all $0 < \tau < \tau_0$:
\begin{align*}
	\| u - \ut_\tau \|_{L^\infty(I,H^1(\D))} \le C_I(u) \, \tau^{q+1}
\end{align*}
where
\begin{align*}
	C_I(u) := C \Big( \| \partial_t^{q+1} u \|_{L^\infty(I,H^1(\D))} + \| \partial_t^{q+2} u \|_{L^\infty(I,H^1(\D))} + \| \partial_t^{q+1} \H u \|_{L^\infty(I,H^1(\D))} \Big)
\end{align*}
for some constant $C > 0$ independent of $\tau_n$ and $u$.
\end{lemma}

\begin{proof}
The error splitting \eqref{errorsplit} and the interpolation estimate \eqref{LagrangeEstimate} give
\begin{align*}
	\| u - \ut_\tau \|_{L^\infty(I,H^1(\D))} & \le C \| \nabla u - \I^\Lo_\tau \nabla u \|_{L^\infty(I,L^2(\D))} + C \| \nabla e_\tau \|_{L^\infty(I,L^2(\D))} \\
	& \le C \tau^{q+1} \| \partial_t^{q+1} \nabla u \|_{L^\infty(I,L^2(\D))} + C \| \nabla e_\tau \|_{L^\infty(I,L^2(\D))}.
\end{align*}
So it remains to estimate the second term on the right hand side. Recall the $q$-stage Gauss-Legendre quadrature rule from \eqref{GLquadrature} and the transformed nodes $t^\GL_{n,j} = t_n + s_j^\GL \tau_n$ and weights $w^\GL_{n,j} = \tau_n w^\GL_j$. We now denote by $\I^\GL_{n,\tau}$ the Lagrange interpolation operator on $I_n$ associated with the nodes $t^\GL_{n,1},\dots,t^\GL_{n,q}$. Then for $v \in \mathcal{P}^q(I_n)$ and $\varphi \in \mathcal{P}^{q-1}(I_n)$ we have by the exactness of the quadrature rule
\begin{align} \label{interpolation-projection}
	\int_{I_n} (\I^\GL_{n,\tau} v) \varphi \d t = \sum_{j = 1}^q w^\GL_{n,j} (\I^\GL_{n,\tau} v)(t^\GL_{n,j}) \varphi(t^\GL_{n,j}) = \sum_{j = 1}^q w^\GL_{n,j} v(t^\GL_{n,j}) \varphi(t^\GL_{n,j}) = \int_{I_n} v \varphi \d t.
\end{align}
Thus,
\begin{align} \label{relation_mij}
\begin{split}
	\Re \sum_{i = 1}^q \sum_{j = 0}^q m_{ij} (e_\tau^{n,j},e_\tau^{n,i})_\H = \Re \int_{I_n} (\partial_t e_\tau, \I^\GL_{n,\tau} e_\tau )_\H \d t & = \Re \int_{I_n} (\partial_t e_\tau,e_\tau)_\H \d t \\
	&  = \frac{1}{2} \| e_\tau^{n+1} \|_{\H}^2 - \frac{1}{2} \| e_\tau^{n} \|_{\H}^2.
\end{split}
\end{align}
So testing the error equation \eqref{erroreq} with $\psi = \H e_\tau^{n,i}$, summing over $i=1,\dots,q$ and taking real parts lead to
\begin{align} \label{eq1}
\begin{split}
	& \frac{1}{2} \| e_\tau^{n+1} \|_{\H}^2 - \frac{1}{2} \| e_\tau^{n} \|_{\H}^2 =  \Im \beta \sum_{i = 1}^q \int_{I_n} \ell_{n,i} (f(\ut_\tau) - f(\I^\Lo_\tau u), e_\tau^{n,i})_\H \d t  \\
	& + \Im \beta \sum_{i = 1}^q \int_{I_n} \ell_{n,i} (f(\I^\Lo_\tau u) - f(u),e_\tau^{n,i})_\H \d t + \Re \sum_{i=1}^q (A^{n,i} - \i B^{n,i},e_\tau^{n,i}))_\H.
\end{split}
\end{align}

From \eqref{W1infty} we have that $\| \I^\Lo_\tau u (t) \|_{W^{1,\infty}(\D)} < \| u \|_{L^\infty(I,W^{1,\infty}(\D))} + 1$ for $t \in I$. Hence the Lipschitz estimate from Lemma \ref{cutoff} and the norm equivalence from Lemma \ref{normequiv} imply
\begin{align*}
	\Big| \sum_{i = 1}^q \int_{I_n} \ell_{n,i} (f(\ut_\tau) & - f(\I^\Lo_\tau u), e_\tau^{n,i})_\H \d t \Big| \le C \| \nabla e_\tau \|_{L^2(I_n \times \D)} \Big( \sum_{i=1}^q \int_{I_n} |\ell_{n,i}|^2 \d t \, \| \nabla e_\tau^{n,i} \|_{L^2(\D)}^2 \Big)^{1/2} \\
	& \le C \| \nabla e_\tau \|_{L^2(I_n \times \D)} \Big( \tau_n \sum_{i=1}^q \| \nabla e_\tau^{n,i} \|_{L^2(\D)}^2 \Big)^{1/2} \le C \| \nabla e_\tau \|_{L^2(I_n \times \D)}^2.
\end{align*}
Further, the interpolation estimate \eqref{LagrangeEstimate} gives
\begin{align*}
	\Big| \sum_{i = 1}^q \int_{I_n} \ell_{n,i} (f(u) - f(\I^\Lo_\tau u), e_\tau^{n,i})_\H \d t  \Big| & \le C \| \nabla (u - \I^\Lo_\tau u) \|_{L^2(I_n \times \D)} \| \nabla e_\tau \|_{L^2(I_n \times \D)} \\
	& \le C \tau_n^{q+1} \| \partial_t^{q+1} \nabla u\|_{L^2(I_n \times \D)} \| \nabla e_\tau \|_{L^2(I_n \times \D)}.
\end{align*}
Using Lemma \ref{ABest} we estimate
\begin{align*}
	\Big|\Re \sum_{i=1}^q (A^{n,i}, e_\tau^{n,i})_\H \Big| & \le C \sum_{i=1}^q \tau_n^{q+1} \| \partial_t^{q+2} \nabla u  \|_{L^2(I_n \times \D)} \tau_n^{1/2} \| \nabla e_\tau^{n,i} \|_{L^2(\D)} \\
	& \le C \tau_n^{2q+2} \| \partial_t^{q+2} \nabla u \|_{L^2(I_n \times \D)}^2 + C \| \nabla e_\tau \|_{L^2(I_n \times \D)}^2
\end{align*}
and similarly
\begin{align*}
	\Big|\Re \sum_{i=1}^q (B^{n,i}, e_\tau^{n,i})_\H \Big| \le C \tau_n^{2q+2} \| \partial_t^{q+1} \nabla \H u \|_{L^2(I_n \times \D)}^2 + C \| \nabla e_\tau \|_{L^2(I_n \times \D)}^2.
\end{align*}
Now \eqref{eq1} and the previous estimates imply
\begin{align} \label{consistent}
\begin{split}
	\| e_\tau^{n+1} \|_{\H}^2 & \le \| e_\tau^n \|_{\H}^2 + C \| \nabla e_\tau \|_{L^2(I_n \times \D)}^2 \\
	& \quad + C\tau_n^{2q + 2} \Big( \| \partial_t^{q+1} \nabla u \|_{L^2(I_n \times \D)}^2 + \| \partial_t^{q+2} \nabla u \|_{L^2(I_n \times \D)}^2 + \| \partial_t^{q+1} \nabla \H u \|_{L^2(I_n \times \D)}^2 \Big) \\
	& \le \| e_\tau^n \|_{\H}^2 + C \| \nabla e_\tau \|_{L^2(I_n \times \D)}^2 + C_n(u) \tau_n^{2q+2}
\end{split}
\end{align}
with $C_n(u)$ from Lemma \ref{consistencyH1}. Taking the local error estimate from Lemma \ref{consistencyH1} into account we obtain by recursion 
\begin{align} \label{recursionL2}
\begin{split}
	\| e_\tau^{n+1} \|_{\H}^2 \le \| e_\tau^n \|_{\H}^2 + C \| \nabla e_\tau \|_{L^2(I_n \times \D)}^2 + C_n(u) \tau_n^{2q+2} & \le \tau^{2q + 2} \sum_{m=0}^n (1+C^2\tau)^{n+1-m} C_m(u) \\
	& \le \tau^{2q + 2} e^{C^2t_{n+1}} \sum_{m=0}^n C_m(u)
\end{split}
\end{align}
and so, in particular,
\begin{align*}
	\| \nabla e_\tau \|_{L^2(I_n \times \D)}^2 \le C\tau_n \Big( \| e_\tau^n \|_{\H}^2 + \tau_n^{2q+3} C_n(u) \Big) \le C\tau_n \Big( \tau_n^{2q+2}  \sum_{m=0}^{n} C_m(u) \Big).
\end{align*}
Finally, the inverse inequality $\| v \|_{L^\infty(I_n)} \le C\tau_n^{-1/2} \| v \|_{L^2(I_n)}$ for $v \in \mathcal{P}^q(I_n)$ implies
\begin{align*}
	& \| \nabla e_\tau \|_{L^\infty(I_n,L^2(\D))}^2 \le C \tau_n^{-1} \| \nabla e_\tau \|_{L^2(I_n \times \D)}^2 \le C\tau_n^{2q+2} \sum_{m=0}^n C_m(u) \\
	& \le C \tau_n^{2q+2} \Big( \| \partial_t^{q+1} \nabla u \|_{L^\infty(I,L^2(\D))}^2 + \| \partial_t^{q+2} \nabla u \|_{L^\infty(I,L^2(\D))}^2 + \| \partial_t^{q+1} \nabla \H u \|_{L^\infty(I,L^2(\D))}^2 \Big).
\end{align*}
Taking square root and the maximum over $n = 0,\dots,N-1$ proves the claim.
\end{proof}

\subsubsection{$L^\infty(I,H^2(\D))$-estimates and $L^\infty(I \times \D)$-bounds}

This section is devoted to $L^\infty(I,H^2(\D))$-error estimates of the truncated semi-discrete approximation satisfying \eqref{semiproblem}. The estimates then imply uniform $L^\infty(\D)$-bounds of the semi-discrete approximation. We start by showing an intermediate result given by a local estimate for $\H e_\tau$ at the temporal nodes $t_n$. For that purpose, we note that the error equation \eqref{erroreq} for $e_\tau$ can be rewritten in the more general form
\begin{align} \label{erroreq0}
\begin{split}
	\int_{I_n} (\partial_t e_\tau, v) - \i (\H e_\tau, v) - \i \beta ( & f(\I^\Lo_\tau u) - f(\ut_\tau), v) \d t =  \int_{I_n} (\partial_t \I^\Lo_\tau u - \partial_t u, v) \\
	& - \i \int_{I_n}(\I^\Lo_\tau \H u - \H u, v) \d t - \i \beta \int_{I_n} (f(\I^\Lo_\tau u) - f(u), v) \d t
\end{split}
\end{align}
for all $v \in \W_{q-1}$. The idea now is to test the equation with $\partial_t \H e_\tau$ which is admissible since $e_\tau(t) \in H^2(\D)$ and then taking imaginary parts. This will result in the following estimate. 

\begin{lemma} \label{consistencyH2}
Under the assumptions \ref{A1}-\ref{A4}, \ref{A7} and \ref{A8} there are constants $C,\tau_0 > 0$ such that for all $0 < \tau_n < \tau_0$ and $n = 0,\dots,N-1$ there hold
\begin{align*}
	\| \H e_\tau^{n+1} \|_{L^2(\D)}^2 \le \| \H e_\tau^{n} \|_{L^2(\D)}^2 + C \tau_n^{-2} \| \nabla e_\tau \|_{L^2(I_n \times \D)}^2 + C_n(u) \tau_n^{2q+2}
\end{align*}
where $C_n(u)$ is defined as in Lemma \ref{consistencyH1}.
\end{lemma}

\begin{proof}
We test in \eqref{erroreq0} with $v = \partial_t \H e_\tau$ and take imaginary parts. Then we obtain
\begin{align} \label{prooflocalH2:eq1}
\begin{split}
	\frac{1}{2} \| \H e_\tau^{n+1} \|_{L^2(\D)}^2 & - \frac{1}{2} \| \H e_\tau^{n} \|_{L^2(\D)}^2 = \Re \int_{I_n} (\H e_\tau, \partial_t \H e_\tau) \d t \\
	& = \beta\, \Re \int_{I_n} ( f(\I^\Lo_\tau u)  - f(\ut_\tau), \partial_t e_\tau)_\H \d t + \Im \int_{I_n} (\partial_t \I^\Lo_\tau u - \partial_t u, \partial_t e_\tau)_\H \d t \\
	& \quad - \Re \int_{I_n} (\I^\Lo_\tau \H u - \H u, \partial_t e_\tau)_\H \d t - \beta\, \Re \int_{I_n} ( f(\I^\Lo_\tau u)  - f(u), \partial_t e_\tau)_\H \d t.
\end{split}
\end{align}
Here we used that $(\partial_t e_\tau(t), \H \partial_t e_\tau(t)) \in \R$ for all $t$. In a similar fashion as in the previous section the terms on the right hand side are estimated using Lemma \ref{cutoff}, Lemma \ref{elliptic} and the interpolation estimate \eqref{LagrangeEstimate},
\begin{align*}
	\Big| \int_{I_n} ( f(\I^\Lo_\tau u)  - f(\ut_\tau), \partial_t e_\tau)_\H \d t \Big| \le C \| \partial_t \nabla e_\tau \|_{L^2(I_n \times \D)}^2,
\end{align*}
and
\begin{align*}
	\Big| \int_{I_n} ( f(\I^\Lo_\tau u)  - f(u), \partial_t e_\tau)_\H \d t \Big| & \le C \int_{I_n} \| \I^\Lo_\tau u - u \|_{H^1(\D)} \| \partial_t \nabla e_\tau \|_{L^2(\D)} \d t \\
	& \le C \tau_n^{2q + 2} \| \partial_t^{q+1} \nabla u\|_{L^2(I_n \times \D)}^2 + C \| \partial_t \nabla e_\tau \|_{L^2(I_n \times \D)}^2.
\end{align*}
And similarly,
\begin{align*}
	\Big| \int_{I_n} ( \partial_t \I^\Lo_\tau u - \partial_t u , \partial_t e_\tau)_\H \d t \Big| \le  C \tau_n^{2q + 2} \| \partial_t^{q+2} \nabla u\|_{L^2(I_n \times \D)}^2 + C \| \partial_t \nabla e_\tau \|_{L^2(I_n \times \D)}^2,
\end{align*}
as well as
\begin{align*}
	\Big| \int_{I_n} ( \I^\Lo_\tau \H u - \H u , \partial_t e_\tau)_\H \d t \Big| \le  C \tau_n^{2q + 2} \| \partial_t^{q+1} \nabla \H u\|_{L^2(I_n \times \D)}^2 + C \| \partial_t \nabla e_\tau \|_{L^2(I_n \times \D)}^2.
\end{align*}
With the inverse estimate $\| \partial_t \nabla e_\tau \|_{L^2(I_n \times \D)} \le C \tau_n^{-1} \| \nabla e_\tau \|_{L^2(I_n \times \D)}$ equation \eqref{prooflocalH2:eq1} now yields
\begin{align*}
	\| \H e_\tau^{n+1} \|_{L^2(\D)}^2 \le \| \H e_\tau^{n+1} \|_{L^2(\D)}^2 + C \tau_n^{-2} \| \nabla e_\tau \|_{L^2(I_n \times \D)}^2 + C_n(u) \tau_n^{2q + 2}.
\end{align*}

\end{proof}

Now we turn back to the error equation \eqref{erroreq} and recall that it holds even for all $\psi \in L^2(\D)$ as we can interpret $(  e_\tau^{n,i}, \psi )_\H = ( \H e_\tau^{n,i} , \psi )$ for $e_\tau^{n,i}\in H^2(\D) \cap H^1_0(\D)$. In contrast to the previous section for the $L^\infty(I,H^1(\D))$-estimates, we test in \eqref{erroreq} with $\psi = \H e_\tau^{n,i} \in L^2(\D)$, summing over $i= 1,\dots,q$ and take the imaginary part to infer the $L^\infty(I,H^2(\D))$-estimates. As we will see, the $L^\infty(I,H^2(\D))$-estimates are non-optimal but sufficient to conclude $L^\infty(I \times \D)$-error bounds of the semi-discrete approximation. This is the main result of this section.

\begin{lemma} \label{semiH2estimate}
Under the assumptions \ref{A1}-\ref{A4}, \ref{A7} and \ref{A8} there is $\tau_0 > 0$ such that for all $0 < \tau < \tau_0$ there hold
\begin{align*}
	\| u - \ut_\tau\|_{L^\infty(I,H^2(\D))} \le C_I(u) \, \tau^{q}
\end{align*}
where $C_I(u)$ is defined as in Lemma \ref{semiH1estimate}. In particular, this implies
\begin{align*}
	\| u - \ut_\tau \|_{L^\infty(I \times \D)} \le C_I(u) \, \tau^{q}.
\end{align*}
\end{lemma}
Note that the latter statement of the lemma expresses that $ \ut_\tau$ is uniformly bounded in $L^{\infty}$ and that the truncation can be dropped for all sufficiently small values of $\tau$.

\begin{proof}
First we note that $e_\tau = \I^{\Lo}_\tau u - \ut_\tau \in L^\infty(I,H^1_0(\D) \cap H^2(\D))$. Then elliptic regularity theory (cf. \cite{GilbargTrudinger01}), Sobolev embedding and the interpolation \eqref{LagrangeEstimate} imply
\begin{align*}
	\| u - \ut_\tau \|_{L^\infty(I \times \D)} \le C \| u - \ut_\tau \|_{L^\infty(I,H^2(\D))} & \le C \| u - \I^\Lo_\tau u \|_{L^\infty(I,H^2(\D))} + C \| e_\tau \|_{L^\infty(I,H^2(\D))} \\
	& \le C \tau^{q + 1} \| \partial_t^{q + 1} u \|_{L^\infty(I,H^2(\D))} + C \| \H e_\tau \|_{L^\infty(I,L^2(\D))}.
\end{align*}
So it remains to estimate $\| \H e_\tau \|_{L^\infty(I,L^2(\D))}$. Testing the error equation \eqref{erroreq} with $\psi = \H e_\tau^{n,i} \in L^2(\D)$ (which is again admissible), summing over $i = 1,\dots,q$ and taking the imaginary part yields
\begin{align*}
	\tau_n \sum_{i=1}^q w^\GL_i \| \H e_\tau^{n,i} \|_{L^2(\D)}^2 & = \Im \sum_{i = 1}^{q} \sum_{j = 0}^{q} m_{ij} ( e_\tau^{n,j}, e_\tau^{n,i})_\H \\
	& \quad - \beta\, \Re \sum_{i = 1}^q \int_{I_n} \ell_{n,i} (f(\I^\Lo_\tau u) - f(\ut_\tau),e_\tau^{n,i})_\H \d t  \\
	& \quad + \beta\, \Re \sum_{i = 1}^q \int_{I_n} \ell_{n,i} (f(\I^\Lo_\tau u) - f(u),e_\tau^{n,i})_\H \d t \\
	& \quad - \Im \sum_{i=1}^q (A^{n,i} - \i B^{n,i},e_\tau^{n,i})_\H.
\end{align*}
Estimating the terms on the right hand side with Lemma \ref{elliptic} and arguments similar as in the proof of Lemma \ref{semiH1estimate} yields
\begin{align*}
		\tau_n \sum_{i = 1}^q w^\GL_i \| \H e_\tau^{n,i} \|_{L^2(\D)}^2 & \le C\tau_n^{-1} \| \nabla e_\tau \|_{L^2(I_n \times \D)}^2 + C_n(u) \tau_n^{2q+2}
\end{align*}
with $C_n(u)$ from Lemma \ref{consistencyH1}. With Lemma \ref{consistencyH1} and Lemma \ref{semiH1estimate} we have
\begin{align*}
	\tau_n \sum_{i = 1}^q w^\GL_i \| \H e_\tau^{n,i} \|_{L^2(\D)}^2 \le C \| \nabla e_\tau^n \|_{L^2(\D)}^2 + C_n(u) \tau_n^{2q+2} \le C\tau^{2q + 2} \sum_{m = 0}^n C_m(u).
\end{align*}
Now Lemma \ref{normequiv}, Lemma \ref{semiH1estimate}, Lemma \ref{consistencyH2} and the inverse estimate $\| v \|_{L^\infty(I_n)} \le C \tau_n^{-1/2} \| v \|_{L^2(I_n)}$ for $v \in \mathcal{P}^{q}(I_n)$ lead to
\begin{align*}
		\| \H e_\tau \|_{L^\infty(I_n,L^2(\D))}^2 \le C \tau_n^{-1} \| \H e_\tau \|_{L^2(I_n \times \D)}^2 & \le C \| \H e_\tau^n \|_{L^2(\D)}^2 + C \sum_{i = 1}^q \| \H e_\tau^{n,i} \|_{L^2(\D)}^2 \\
		& \le C \tau^{2q} \sum_{m = 0}^n C_m(u). 
\end{align*}
Taking square root and the maximum over $n = 0,\dots,N-1$ proves the claim.
\end{proof}

\sect{Proof of uniform $L^\infty(\D)$-bounds of the fully-discrete scheme}
\label{section:proof-main-result}

In this section we prove our main result, Theorem \ref{mainresult}. In fact, we show the statement of the theorem for the truncated approximation $\ut_{h,\tau}$ from where we infer that the truncated approximation from \eqref{auxproblem} also solves \eqref{cGscheme}, i.e. $\ut_{h,\tau} = u_{h,\tau}$ in $I \times \D$. Hence the uniform $L^\infty$-bounds also hold for $u_{h,\tau}$. The bound for $\ut_{h,\tau}$ is concluded from $L^\infty$-estimates of the error $u - \ut_{h,\tau}$ which is done by splitting the error into
\begin{align} \label{fullerrorequation}
	u - \ut_{h,\tau} = (u - \ut_{\tau}) + (\ut_\tau - P_h \ut_\tau) + (P_h \ut_\tau - \ut_{h,\tau}) = e_\tau + (\ut_\tau - P_h \ut_\tau) + e_{h,\tau}
\end{align}
with $e_{h,\tau} := P_h \ut_{\tau} - \ut_{h,\tau}$ and $e_\tau$ defined as in \eqref{errorsplit}. The first term in \eqref{fullerrorequation} was estimated in the previous section w.r.t. different norms and in particular w.r.t. the $L^\infty(I \times \D)$-norm. The second term in \eqref{fullerrorequation} is easily estimated due to \ref{A7}, so that
\begin{align*}
	\| \ut_\tau - P_h \ut_\tau \|_{L^\infty(I \times \D)} \le C \| \ut_\tau \|_{L^\infty(I,H^2(\D))} \le C \| e_\tau \|_{L^\infty(I,H^2(\D))} + C \| u \|_{L^\infty(I,H^2(\D))}
\end{align*}
which is uniformly bounded for $0 < \tau < \tau_0$ due to Theorem \ref{semiH2estimate}. Therefore we focus in the following on $e_{h,\tau}$. Recall that $P_h$ is the Ritz-projection of $H^1_0(\D)$ onto $S_h$ w.r.t. the sesquilinear form $(\cdot,\cdot)_\H$. Therefore, we have
\begin{align*}
	& \sum_{j = 0}^q m_{ij} (P_h \ut_\tau^{n,j},\psi) - \i \tau_n w_i^\GL (P_h \ut_\tau^{n,i}, \psi )_\H = \sum_{j = 0}^q m_{ij} (P_h \ut_\tau^{n,j} - \ut_\tau^{n,j}, \psi) + \i \beta \int_{I_n} \ell_{n,i} (f(\ut_\tau),\psi) \d t
\end{align*}
for all $i = 1,\dots,q$, $n = 0,\dots,N-1$ and $\psi \in S_h$. Then the error equation for $e_{h,\tau}$ reads as
\begin{eqnarray} \label{erroreqfull}
\lefteqn{ \sum_{j =0}^q m_{ij} (e_{h,\tau}^{n,j}, \psi) - \i \tau_n w_i^\GL (e_{h,\tau}^{n,i}, \psi)_\H } \\
\nonumber&=& \sum_{j = 0}^q m_{ij} (P_h \ut_\tau^{n,j} - \ut_\tau^{n,j}, \psi ) + \i \beta \int_{I_n} \ell_{n,i} (f(\ut_\tau) - f(\ut_{h,\tau}), \psi) \d t
\end{eqnarray}
for all $i = 1,\dots,q$, $n = 0,\dots,N-1$ and $\psi \in S_h$. Here we use the notation $e_{h,\tau}^n := e_{h,\tau}(t_n)$ and $e_{h,\tau}^{n,i} := e_{h,\tau}(t_{n,i}^\GL)$ analogously as in the previous sections. Using the error equation we now show local error estimates for $e_{h,\tau}$ in $L^2(\D)$.

\begin{lemma} \label{consistencyfull}
Under the assumptions \ref{A1}-\ref{A8} there are constants $C, \tau_0, h_0 > 0$ such that for all $0 < \tau_n < \tau_0$, $0 < h < h_0$ and $n = 0,\dots,N-1$ there hold
\begin{align*}
	\| e_{h,\tau} \|_{L^2(I_n \times \D)}^2 \le C \tau_n \| e_{h,\tau}^{n} \|_{L^2(\D)}^2 + C h^4 \tau_n^2 \Big( C_I(u)^2 \tau^{2q-1} + \| \partial_t u \|_{L^2(I_n,H^2(\D))}^2 + \| u \|_{L^2(I_n,H^2(\D))}^2 \Big),
\end{align*}
where $C_I(u)$ is defined as in Lemma \ref{semiH1estimate}.
\end{lemma}

\begin{proof}
We define $\mathbf{e}_{h,\tau}^{n,j} = (s_j^\GL)^{-\frac{1}{2}} e_{h,\tau}^{n,j}$, $j = 1,\dots,q$. Then testing the error equation \eqref{erroreqfull} with $\psi = (s_i^\GL)^{-\frac{1}{2}} \mathbf{e}_{h,\tau}^{n,i}$, summing over $i = 1,\dots,q$ and taking real part gives
\begin{align} \label{proofconsitencyfull:eq1}
\begin{split}
	& \Re \sum_{i,j = 1}^q \mathbf{m}_{ij} (\mathbf{e}_{h,\tau}^{n,j}, \mathbf{e}_{h,\tau}^{n,i}) =  \Re \sum_{i = 1}^q \sum_{j=0}^q (s_i^\GL)^{-\frac{1}{2}} m_{ij} (P_h \ut_\tau^{n,j} - \ut_\tau^{n,j},\mathbf{e}_{h,\tau}^{n,i} ) \\
	& \quad - \Im \beta \sum_{i = 1}^q \int_{I_n} (s_i^\GL)^{-\frac{1}{2}} \ell_{ni} (f(\ut_\tau) - f(\ut_{h,\tau}), \mathbf{e}_{h,\tau}^{n,i} ) \d t - \Re \sum_{i = 1}^q (s_i^\GL)^{-\frac{1}{2}} m_{i0} (e_{h,\tau}^n,\mathbf{e}_{h,\tau}^{n,i}). \\
\end{split}
\end{align}
For the left hand side in \eqref{proofconsitencyfull:eq1} we find with Lemma \ref{elliptic} and Lemma \ref{stability} a constant $\alpha > 0$ such that
\begin{align*}
	\alpha \sum_{i = 1}^q \| \mathbf{e}_{h,\tau}^{n,i} \|_{L^2(\D)}^2 \le \Re \sum_{i,j = 1}^q \mathbf{m}_{ij} (\mathbf{e}_{h,\tau}^{n,j}, \mathbf{e}_{h,\tau}^{n,i}).
\end{align*}
Next we consider the right hand side of \eqref{proofconsitencyfull:eq1}. Let us write $\eta : = P_h \ut_\tau - \ut_\tau$ and note that by the exactness of the Gauss-Lobatto quadrature we have
\begin{align*}
	\ell_{n,i}(t_{n+1}) - \sum_{j = 0}^{q} w_{n,j}^\Lo \, \ell'_{n,i}(t^\Lo_{n,j}) - \ell_{n,i}(t_{n}) = \ell_{n,i}(t_{n+1}) - \int_{I_n} \ell_{n,i}'(t) \d t - \ell_{n,i}(t_{n}) = 0.
\end{align*}
Then it is straightforward to construct constants $\beta_{ij} \in \R$ (depending on the Gauss-Lobatto quadrature points and weights) such that 
\begin{align*}
	\ell_{n,i}(t_{n+1}) \eta(t_{n+1}) - \sum_{j = 0}^{q} w_{n,j}^\Lo \, \ell'_{n,i}(t^\Lo_{n,j}) \, \eta(t^\Lo_{n,j}) - \ell_{n,i}(t_n) \eta(t_n)
		& = \sum_{j = 1}^q \beta_{ij} (\eta(t^\Lo_{n,j}) - \eta(t^\Lo_{n,j-1})).
\end{align*}
Now using again the exactness of the Gauss-Lobatto quadrature we have
\begin{align*}
	 \sum_{j=0}^q m_{ij} (P_h \ut_\tau^{n,j} - \ut_\tau^{n,j}) & = \int_{I_n} \ell_{n,i}(t) \, \partial_t \eta(t) \d t \\
		& = \ell_{n,i}(t_{n+1}) \eta(t_{n+1}) - \int_{I_n} \ell'_{n,i}(t) \, \eta(t) \d t - \ell_{n,i}(t_n) \eta(t_n) \\
		& = \ell_{n,i}(t_{n+1}) \eta(t_{n+1}) - \sum_{j = 0}^{q} w_{n,j}^\Lo \, \ell'_{n,i}(t^\Lo_{n,j}) \, \eta(t^\Lo_{n,j}) - \ell_{n,i}(t_n) \eta(t_n) \\
		& = \sum_{j = 1}^q \beta_{ij} (\eta(t^\Lo_{n,j}) - \eta(t^\Lo_{n,j-1})) = \sum_{j = 1}^q \beta_{ij} (P_h - I)(\ut_\tau(t^\Lo_{n,j}) - \ut_\tau(t^\Lo_{n,j-1})).
\end{align*}
So using \ref{A5} and Lemma \ref{semiH2estimate} we estimate
\begin{align*}
		& \sum_{j = 0}^q \| (P_h-I)(\ut_\tau(t^\Lo_{n,j}) - \ut_\tau(t^\Lo_{n,j-1})) \|_{L^2(\D)} \le Ch^2 \sum_{j = 0}^q | \ut_\tau(t^\Lo_{n,j}) - \ut_\tau(t^\Lo_{n,j-1}) |_{H^2(\D)} \\
		& \le Ch^2 \sum_{j = 0}^q \Big( |\ut_\tau(t^\Lo_{n,j}) - u(t^\Lo_{n,j})|_{H^2(\D)} + |\ut_\tau(t^\Lo_{n,j-1}) - u(t^\Lo_{n,j-1})|_{H^2(\D)} + \int_{t^\Lo_{n,j-1}}^{t^\Lo_{n,j}} |\partial_t u|_{H^2(\D)} \d t \Big) \\
		& \le C_I(u) h^2 \tau_n^{q} + Ch^2 \int_{I_n} \|\partial_t u \|_{H^2(\D)} \d t \le C_I(u) h^2 \tau_n^{q} + Ch^2 \tau_n^{1/2} \| \partial_t u \|_{L^2(I_n,H^2(\D))}.
\end{align*}
This yields for the first term of the right hand side in \eqref{proofconsitencyfull:eq1} the estimate
\begin{align} \label{proofconsitencyfull:eq2}
\begin{split}
	& \Big| \sum_{i = 1}^q \sum_{j=0}^q (s_i^\GL)^{-\frac{1}{2}} m_{ij} (P_h \ut_\tau^{n,j} - \ut_\tau^{n,j},\mathbf{e}_{h,\tau}^{n,i} ) \Big| \le C \sum_{i = 1}^q \sum_{j = 0}^q \| P_h \ut_\tau^{n,j} - \ut_\tau^{n,j} \|_{L^2(\D)} \| \mathbf{e}_{h,\tau}^{n,i} \|_{L^2(\D)} \\
	& \le \Big( C_I(u) h^2 \tau_n^{q} + Ch^2 \tau_n^{1/2} \| \partial_t u \|_{L^2(I_n,H^2(\D))} \Big) \Big( \sum_{i = 1}^q \| \nabla \mathbf{e}_{h,\tau}^{n,i} \|_{L^2(\D)}^2 \Big)^{1/2}.
\end{split}
\end{align}
Next we estimate the second term in \eqref{proofconsitencyfull:eq1} using Lemma \ref{cutoff} by
\begin{align} \label{proofconsitencyfull:eq3}
\begin{split}
	& \Big| \beta \sum_{i = 1}^q \int_{I_n} (s_i^\GL)^{-\frac{1}{2}} \ell_{ni} (f(\ut_\tau) - f(\ut_{h,\tau}), \mathbf{e}_{h,\tau}^{n,i} ) \d t \Big| \\
	& \le C\int_{I_n} \| \ut_\tau - \ut_{h,\tau} \|_{L^2(\D)} \d t \left( \sum_{i =1}^q  \| \mathbf{e}_{h,\tau}^{n,i} \|_{L^2(\D)} \right) \\
	& \le C \int_{I_n} \| \ut_\tau - P_h \ut_\tau \|_{L^2(\D)} + \| e_{h,\tau} \|_{L^2(\D)} \d t \, \left(\sum_{i =1}^q  \| \mathbf{e}_{h,\tau}^{n,i} \|_{L^2(\D)}  \right)\\
	& \le C \Big( h^2 \int_{I_n} | \ut_\tau |_{H^2(\D)} \d t + \tau_n^{1/2} \| e_{h,\tau} \|_{L^2(I_n \times \D)} \Big) \left( \sum_{i =1}^q  \| \mathbf{e}_{h,\tau}^{n,i} \|_{L^2(\D)} \right) \\
	& \le C \Big( h^2 \int_{I_n} | \ut_\tau - u |_{H^2(\D)} + | u |_{H^2(\D)} \d t + \tau_n^{1/2} \| e_{h,\tau} \|_{L^2(I_n \times \D)} \Big) \left( \sum_{i =1}^q  \| \mathbf{e}_{h,\tau}^{n,i} \|_{L^2(\D)} \right) \\
	& \le C \tau_n^{1/2} \Big( C_I(u) h^2 \tau^{q+\frac{1}{2}} + h^2 \| u \|_{L^2(I_n,H^2(\D))} + \| e_{h,\tau} \|_{L^2(I_n \times \D)} \Big) \Big( \sum_{i =1}^q  \| \mathbf{e}_{h,\tau}^{n,i} \|_{L^2(\D)}^2 \Big)^{1/2}.
\end{split}
\end{align}
The third term in \eqref{proofconsitencyfull:eq1} is estimated by
\begin{align*}
	\Big| \sum_{i = 1}^q (s_i^\GL)^{-\frac{1}{2}} m_{i0} (e_{h,\tau}^n,\mathbf{e}_{h,\tau}^{n,i}) \Big| \le C \| e_{h,\tau}^n \|_{L^2(\D)} \Big( \sum_{i = 1}^q \| \mathbf{e}_{h,\tau}^{n,i} \|_{L^2(\D)} \Big)^{1/2}.
\end{align*}
Summarizing we have shown
\begin{align*}
	\Big( \sum_{i = 1}^q \| \mathbf{e}_{h,\tau}^{n,i} \|_{L^2(\D)}^2 \Big)^{1/2} & \le C\| e_{h,\tau}^n \|_{L^2(\D)} + C \tau_n^{1/2} \| e_{h,\tau} \|_{L^2(I_n \times \D)} \\
	& \quad + Ch^2 \tau_n^{1/2} \Big( C_I(u) \tau^{q - \frac{1}{2}} + \| \partial_t u \|_{L^2(I_n,H^2(\D))} + \|  u \|_{L^2(I_n,H^2(\D))} \Big).
\end{align*}
Taking Lemma \ref{normequiv} into account gives
\begin{align*}
	\| e_{h,\tau} \|_{L^2(I_n \times \D)}^2 & \, \le \, C \tau_n \sum_{i = 1}^q \| e_{h,\tau}^{n,i} \|_{L^2(\D)}^2  \, \le \, C \tau_n \sum_{i = 1}^q \| \mathbf{e}_{h,\tau}^{n,i} \|_{L^2(\D)}^2 \\
	&\,  \le \, C \tau_n \| e_{h,\tau}^n \|_{L^2(\D)}^2 + C\tau_n^2 \| e_{h,\tau} \|_{L^2(I_n \times \D)}^2 \\
	& \quad + C h^4 \tau_n^2 \Big( C_I(u)^2 \tau^{2q-1} + \| \partial_t u \|_{L^2(I_n,H^2(\D))}^2 + \| u \|_{L^2(I_n,H^2(\D))}^2 \Big).
\end{align*}
Hence, taking $\tau$ sufficiently small yields the claim.
\end{proof}

The result of the previous lemma is now used to show global estimates for the error $e_{h,\tau}$.

\begin{lemma} \label{fullestimate}
Under the assumptions \ref{A1}-\ref{A8} there are constants $C,\tau_0,h_0 > 0$ such that for all $0 < \tau < \tau_n$ and $0 < h < h_0 $ there hold
\begin{align*}
	\| e_{h,\tau} \|_{L^\infty(I,L^2(\D))} \le C h^2 \Big( C_I(u) \, \tau^{q-1} + \| u \|_{L^\infty(I,H^2(\D))} + \| \partial_t u\|_{L^\infty(I,H^2(\D))} \Big)
\end{align*}
where $C_I(u)$ is defined as in Lemma \ref{semiH1estimate}.
\end{lemma}

\begin{proof}
Testing the error equation \eqref{erroreqfull} with $\psi = e_{h,\tau}^{n,i}$, summing over $i = 1,\dots,q$ and taking real part gives analogously to \eqref{relation_mij}:
\begin{align} \label{prooffullesimate:eq1}
\begin{split}
	& \frac{1}{2} \| e_{h,\tau}^{n+1} \|_{L^2(\D)}^2 - \frac{1}{2} \| e_{h,\tau}^{n} \|_{L^2(\D)}^2 = \Re \sum_{i = 1}^q \sum_{j = 0}^q m_{ij} (e_{h,\tau}^{n,j}, e_{h,\tau}^{n,i}) \\
	& = \Re \sum_{i = 1}^q \sum_{j = 0}^q m_{ij} (P_h \ut_\tau^{n,j} - \ut_\tau^{n,j}, e_{h,\tau}^{n,i}) - \beta\, \Im \sum_{i=1}^q \int_{I_n} \ell_{n,i} (f(\ut_\tau) - f(\ut_{h,\tau}), e_{h,\tau}^{n,i}) \d t.
\end{split}
\end{align}
As in \eqref{proofconsitencyfull:eq2} and \eqref{proofconsitencyfull:eq3} we have
\begin{align*}
	\Big| \sum_{i = 1}^q \sum_{j = 0}^q m_{ij} (P_h \ut_\tau^{n,j} - \ut_\tau^{n,j}, e_{h,\tau}^{n,i}) \Big| & \le \Big( C_I(u) h^2 \tau^{q - \frac{1}{2}} + Ch^2 \| \partial_t u \|_{L^2(I_n,H^2(\D))} \Big) \Big( \tau_n \sum_{i = 1}^q \| e_{h,\tau}^{n,i} \|_{L^2(\D)}^2 \Big)^{1/2} \\
	& \le C_I(u)^2 h^4 \tau^{2q-1} + Ch^4 \| \partial_t u \|_{L^2(I_n,H^2(\D))}^2 + C \| e_{h,\tau} \|_{L^2(I_n \times \D)}^2
\end{align*}
with $C_I(u)$ from Lemma \ref{semiH1estimate} and
\begin{align*}
	& \Big| \beta \sum_{i=1}^q \int_{I_n} \ell_{n,i} (f(\ut_\tau) - f(\ut_{h,\tau}), e_{h,\tau}^{n,i}) \d t  \Big| \\
	& \le \Big( C_I(u) h^2 \tau^{q} + Ch^2 \| u \|_{L^2(I_n,H^2(\D))} + C \| e_{h,\tau} \|_{L^2(I_n \times \D)} \Big) \Big( \tau_n \sum_{i =1}^q  \| e_{h,\tau}^{n,i} \|_{L^2(\D)}^2 \Big)^{1/2} \\
	& \le C_I(u)^2 h^4 \tau^{2q} + Ch^4 \| u \|_{L^2(I_n,H^2(\D))}^2 + C \| e_{h,\tau} \|_{L^2(I_n \times \D)}^2.
\end{align*}
Hence \eqref{prooffullesimate:eq1} implies with the local error estimate from Lemma \ref{consistencyfull}
\begin{align*}
	\| e_{h,\tau}^{n+1} \|_{L^2(\D)}^2 & \le \| e_{h,\tau}^{n} \|_{L^2(\D)}^2 + C \| e_{h,\tau} \|_{L^2(I_n \times \D)}^2 \\
	& \quad + Ch^4 \Big( C_I(u)^2 \tau^{2q-1} + \| u \|_{L^2(I_n,H^2(\D))}^2 + \| \partial_t u \|_{L^2(I_n,H^2(\D))}^2 \Big) \\
	& \le (1 + C \tau_n) \| e_{h,\tau}^n \|_{L^2(\D)} + Ch^4 \Big( C_I(u)^2 \tau^{2q-1} + \| u \|_{L^2(I_n,H^2(\D))}^2 + \| \partial_t u \|_{L^2(I_n,H^2(\D))}^2 \Big).
\end{align*}
This recursion yields
\begin{align*}
	\| e_{h,\tau}^n \|_{L^2(\D)}^2 & \le Ch^4 (1 + C\tau)^n \sum_{m = 0}^{n-1} \Big( C_I(u)^2 \tau^{2q-1} + \| u \|_{L^2(I_m,H^2(\D))}^2 + \| \partial_t u \|_{L^2(I_m,H^2(\D))}^2 \Big) \\
	& \le C h^4 e^{C T} \Big( C_I(u)^2 \tau^{2q-2} + \| u \|_{L^2(I,H^2(\D))}^2 + \| \partial_t u \|_{L^2(I,H^2(\D))}^2 \Big).
\end{align*}
So by Lemma \ref{normequiv}, Lemma \ref{consistencyfull} and the inverse estimate $\| v \|_{L^\infty(I_n)} \le C \tau_n^{-1/2} \| v \|_{L^2(I_n)}$ for $v \in \mathcal{P}^{q}(I_n)$ we now obtain
\begin{align*}
	\| e_{h,\tau} \|_{L^\infty(I_n,L^2(\D))} & \le C \tau_n^{-1/2} \| e_{h,\tau} \|_{L^2(I_n \times \D)} \\
	& \le C \| e_{h,\tau}^n \|_{L^2(\D)} + Ch^2 \Big( C_I(u) \tau^{q} + \| u \|_{L^2(I_n,H^2(\D))} + \| \partial_t u \|_{L^2(I_n,H^2(\D))} \Big) \\
	& \le C h^2 \Big( C_I(u) \tau^{q-1} + \| u \|_{L^\infty(I,H^2(\D))} + \| \partial_t u \|_{L^\infty(I,H^2(\D))} \Big)
\end{align*}
which yields the claim by taking the maximum over $n = 0,\dots,N-1$.
\end{proof}

Finally we can prove our main result.

\begin{proof}[Proof of Theorem \ref{mainresult}]
We prove the uniform $L^\infty(\D)$-bound for the truncated approximation $\ut_{h,\tau}$. In particular, we show
\begin{align} \label{boundtrucated}
	\| \ut_{h,\tau} \|_{L^\infty(I \times \D)} < M.
\end{align}
Then by Lemma \ref{cutoff} the truncated approximation $\ut_{h,\tau}$ also solves \eqref{cGscheme}, i.e. $\ut_{h,\tau} = u_{h,\tau}$, which then proves the assertion. In order to show \eqref{boundtrucated}, we bound the error $u - \ut_{h,\tau}$ by three terms:
\begin{align*}
	\| u - \ut_{h,\tau} \|_{L^\infty(I \times \D)} \le \| u - \ut_{\tau} \|_{L^\infty(I \times \D)} + \| \ut_\tau - P_h \ut_{\tau} \|_{L^\infty(I \times \D)} + \| e_{h,\tau} \|_{L^\infty(I \times \D)}.
\end{align*}
Lemma \ref{semiH2estimate} shows for the first term the estimate
\begin{align*}
	\| u - \ut_{\tau} \|_{L^\infty(I \times \D)} \le C_I(u) \tau^{q}.
\end{align*}
The second term is bounded due to \ref{A7}, \ref{A8} and Lemma \ref{semiH2estimate} by
\begin{align*}
	\| \ut_\tau - P_h \ut_{\tau} \|_{L^\infty(I \times \D)} \le C_\infty \| \ut_\tau \|_{L^\infty(I,H^2(\D))} & \le C_\infty \| u \|_{L^\infty(I,H^2(\D))} + C_\infty \| u - \ut_\tau \|_{L^\infty(I,H^2(\D))} \\
	& \le C_\infty \| u \|_{L^\infty(I,H^2(\D))} + C_I(u) \tau^{q}.
\end{align*}
Using the inverse inequality \ref{A6} we obtain with Lemma \ref{fullestimate} 
\begin{align*}
	\| e_{h,\tau} \|_{L^\infty(I \times \D)} & \le C h^{-\frac{d}{2}} \| e_{h,\tau} \|_{L^\infty(I,L^2(\D))} \\
	& \le C h^{2 - \frac{d}{2}} \Big( C_I(u) \tau^{q-1} + \| u \|_{L^\infty(I,H^2(\D))} + \| \partial_t u \|_{L^\infty(I,H^2(\D))} \Big).
\end{align*}
This yields the bound
\begin{align*}
 \| \ut_{h,\tau} \|_{L^\infty(I \times \D)} & \le \| u \|_{L^\infty(I \times \D)} + \| u - \ut_{h,\tau} \|_{L^\infty(I \times \D)} \\
 & \le \| u \|_{L^\infty(I \times \D)} + C_\infty \| u \|_{L^\infty(I,H^2(\D))} + 2 C_I(u) \tau^{q} \\
 & \quad + C h^{2 - \frac{d}{2}} \Big( C_I(u) \tau^{q-1} + \| u \|_{L^\infty(I,H^2(\D))} + \| \partial_t u \|_{L^\infty(I,H^2(\D))} \Big) \\
 & < \| u \|_{L^\infty(I \times \D)} + C_\infty \| u \|_{L^\infty(I,H^2(\D))} + 1 = M
\end{align*}
if $d \le 3$ and for sufficiently small $\tau$ and $h$. Hence \eqref{boundtrucated} holds and the assertion is proved. 
\end{proof}

We close this section with a uniqueness result for the fully-discrete approximation satisfying \eqref{cGscheme}. In particular, a numerical solution satisfying the uniform bound from Theorem \ref{mainresult} is unique. More precise, the family of solutions $u_{h,\tau}$ for $0 < \tau < \tau_0$, $0 < h < h_0$ satisfying the uniform bound from Theorem \ref{mainresult} is unique and any other family of solutions of \eqref{cGscheme} must blowup on $I$ w.r.t. the $L^\infty(\D)$-norm as $\tau \rightarrow 0$. This is the result of the following lemma and its proof is given in the appendix.

\begin{lemma} \label{uniqueness}
Let the assumptions \ref{A1}-\ref{A7} be fulfilled and $u_0 \in H^1_0(\D)$. Then for every $M_0 > 0$ there is $\tau_0 > 0$ such that the following statement holds true: If $u_1, u_2 \in V_q$ are two bounded solutions of the fully-discrete problem \eqref{cGscheme} for some $0 < \tau < \tau_0$, so that $\| u_i \|_{L^\infty(I \times \D)} \le M_0$, $i = 1,2$, then $u_1 = u_2$.
\end{lemma}

\sect{Numerical experiments}
\label{sec:numerics}

This section is devoted to numerical experiments illustrating the performance of the presented cG($q$)-methods for the time discretization in the case of the rotating GPE for $d = 2$. In the experiments the cG($q$)-method is combined with a standard $\mathcal{P}^1$-Lagrange finite element discretization in space. We investigate the convergence behavior w.r.t. the time discretization and its approximation quality over larger time scales. This is done by simulating the dynamics of a rotating BEC in a particular model with an anisotropic trapping potential. More precise, we consider the GPE
\begin{align} \label{numericalproblem}
	\i \partial_t u = - \Delta u - \Omega \L_z u + Vu + \beta |u|^2 u, \quad \Omega = 1.6 \quad \text{and} \quad \beta = 200,
\end{align}
on a square $\D = (-R,R)^2$ with radius $R > 0$ and for times $t \in [0,T]$. We assume that the condensate is trapped in an anisotropic potential given by
\begin{align} \label{parameter2}
	V(\x) = (\gamma_x x)^2 + (\gamma_y y)^2, \quad \gamma_x = 0.9, \quad \gamma_y = 1.1.
\end{align}
For the initial value at $t = 0$ we choose the ground sate with a center vortex that minimizes the energy $E_0(u) := \frac{1}{2} \int_{\D} |\nabla u|^2 - \Omega \, \overline{u} \, \L_z u + V_0 |u|^2 + \frac{\beta}{2}|u|^4 \d\x$ with an isotropic quadratic trapping potential $V_0(\x) = x^2 + y^2$. Due to the angular velocity of the condensate and its behavior as a superfluid the ground state develops quantized vortices (density singularities). Loosely speaking, the probability of finding a particle of the condensate at the center of the vortex (almost) vanishes. The vortices are depicted in Figure \ref{BECsim} and we exemplarily refer to \cite{Aftalion2006,BWM05,Fetter01} for more details on this phenomena. In all our experiments the ground state is computed using a generalized inverse iteration as e.g. formulated in \cite{BaoDu04,HenningPeterseim20}. On the computational domain $\D = (-R,R)^2$ we choose a uniform triangulation consisting of $2N_h^2$ rectangular triangles, where $N_h \in \N$. The space $S_h$ is defined to be the space of $H^1$-conforming $\mathcal{P}^1$-Lagrange finite elements associated with the triangulation. Hence $S_h$ consists of $(N_h - 1)^2$ degrees of freedom and with $h = 2R/N_h$, what we refer to as the mesh size. Hence, our assumptions \ref{A5}-\ref{A7} are satisfied (cf. \cite{BrennerScott08}). For the time discretization we use the presented cG($q$)-methods with equidistant time steps of size $\tau > 0$. The nonlinear equations in each time step are solved using a fixed point iteration with a tolerance of magnitude $\varepsilon_{\mathrm{FP}} > 0$ as described in the end of section \ref{section:reformulation}.
\begin{figure}[H]
	\centering
	\includegraphics[scale=0.5]{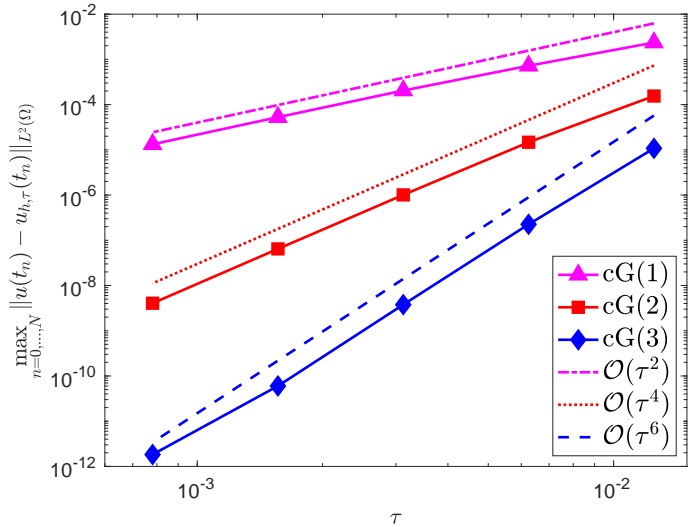}
    \caption{Errors of the cG($q$)-methods applied to the rotating GPE \eqref{numericalproblem}, \eqref{parameter2} on the domain $\D = (-20,20)^2$ with $T = 0.1$, spatial mesh size $h = 40/128$, time step sizes $\tau = T/2^{i}$, $i = 3,\dots,7$ and $q \in \{ 1,2,3 \}$.}
    \label{plot_convergence}
\end{figure}

In our first numerical experiment we investigate the convergence of the cG($q$)-method w.r.t. time for the cases $q = 1,2,3$ and on short time scales for clear experimental convergence rates. For that purpose we fix $R = 20$, $T = 0.1$, $h = 40/128$ and solve the problem \eqref{numericalproblem}, \eqref{parameter2} with time step sizes $\tau = T/2^i$, $i = 3,\dots,7$ for the three cases $q = 1,2,3$. Within the experiments the nonlinear equations in each time step are solved with a fixed point iteration up to a precision of $\varepsilon_{\mathrm{FP}} = 10^{-12}$ so that it does not effect the convergence results. As a reference solution we choose the result of the cG($3$)-method with a fine time step size $\tau = T/2^{10}$ and we calculate the error w.r.t. $L^2(\D)$ at the temporal nodes $t_n$. The results are shown in Figure \ref{plot_convergence} and, as expected, we observe convergence towards the (reference) solution of the problem \eqref{numericalproblem}, \eqref{parameter2}. In particular, we obtain numerically the superconvergence of order $2q$ of the cG($q$)-method. Although a corresponding result has not yet been proved in the general setting of the rotating GPE (i.e. $V \neq 0$ and $\Omega \neq 0$) this gives strong evidence that the superconvergence also holds true in that case. In particular, this points out the advantage of the cG($q$)-methods which can reach (at least theoretically) arbitrary high order in time and therefore allows for improvements in the time stepping compared to standard energy-preserving time integrators such as the Crank-Nicolson scheme which is only convergent of order $2$. \\
In the second experiment we solve again the time-dependent GPE problem \eqref{numericalproblem}, \eqref{parameter2} on a larger time scale with $T = 50$, but on a smaller spatial domain $\D = (-10,10)^2$ (i.e. $R = 10$) to keep the computations feasible. In this experiment, we want to track the longtime dynamics of the solution describing the evolution of a BEC. From the physical point of view, such a setting models a BEC with a constant angular momentum around the $z$-axis that is proportional to the angular velocity $\Omega$ and which is initially prepared as a ground state within an isotropic potential. Then at $t = 0$ a switch is turned on which changes the trapping potential from the isotropic to the anisotropic potential given by $V$. Due to the angular momentum and the change of the trapping potential the BEC is expected to show dynamics within the domain since the system is slightly perturbed w.r.t. the trapping potential. In particular, we expect that the BEC is deformed and the vortices move around in an unpredictable manner. This interpretation of course also applies to our first experiment, but this time we have stronger dynamics due to the larger time scale. \\

\begin{figure}[H]
\flushleft
\begin{minipage}[t]{0.3\textwidth}
\centering
\includegraphics[scale=0.27]{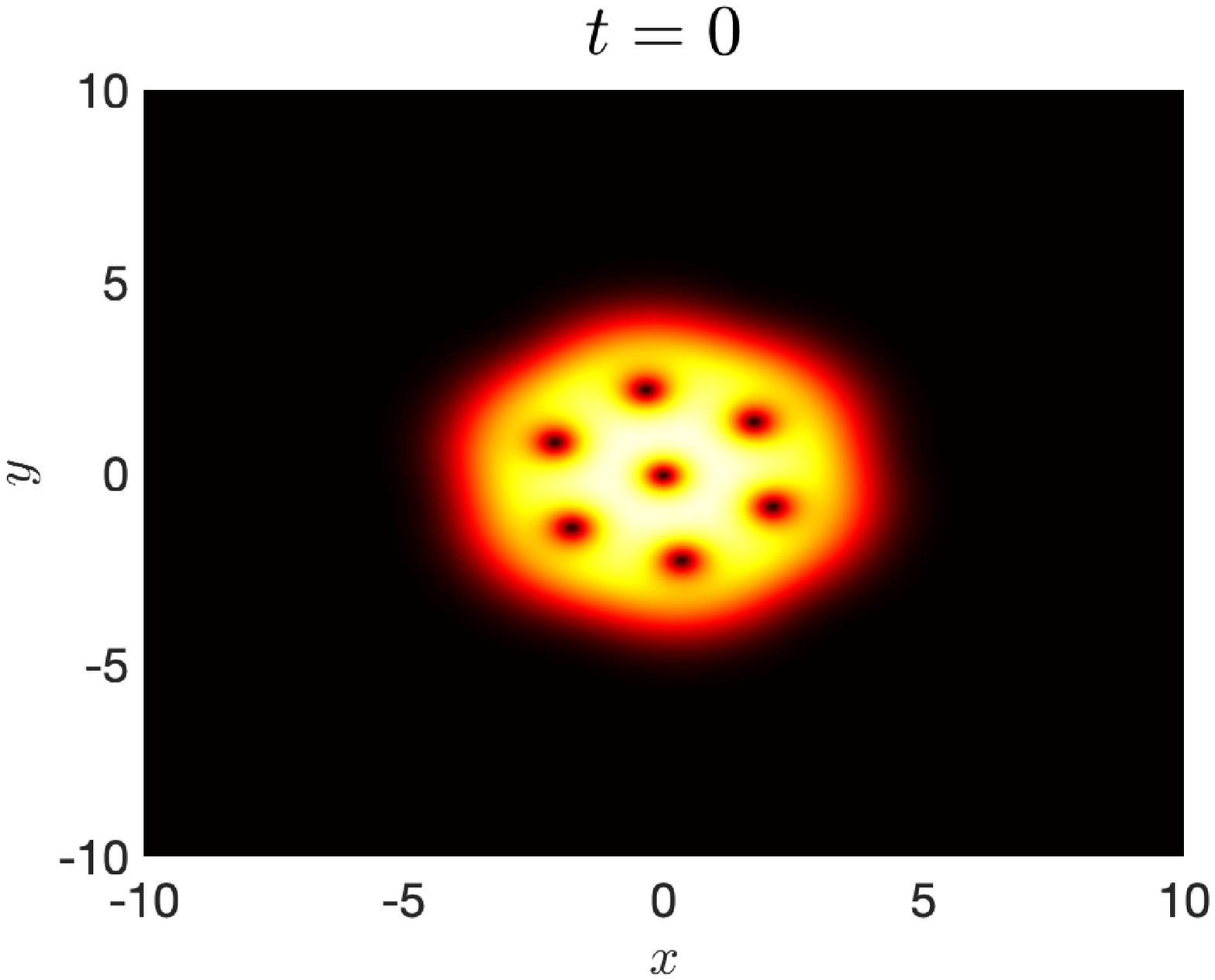}
\end{minipage}
\begin{minipage}[t]{0.3\textwidth}
\centering
\includegraphics[scale=0.27]{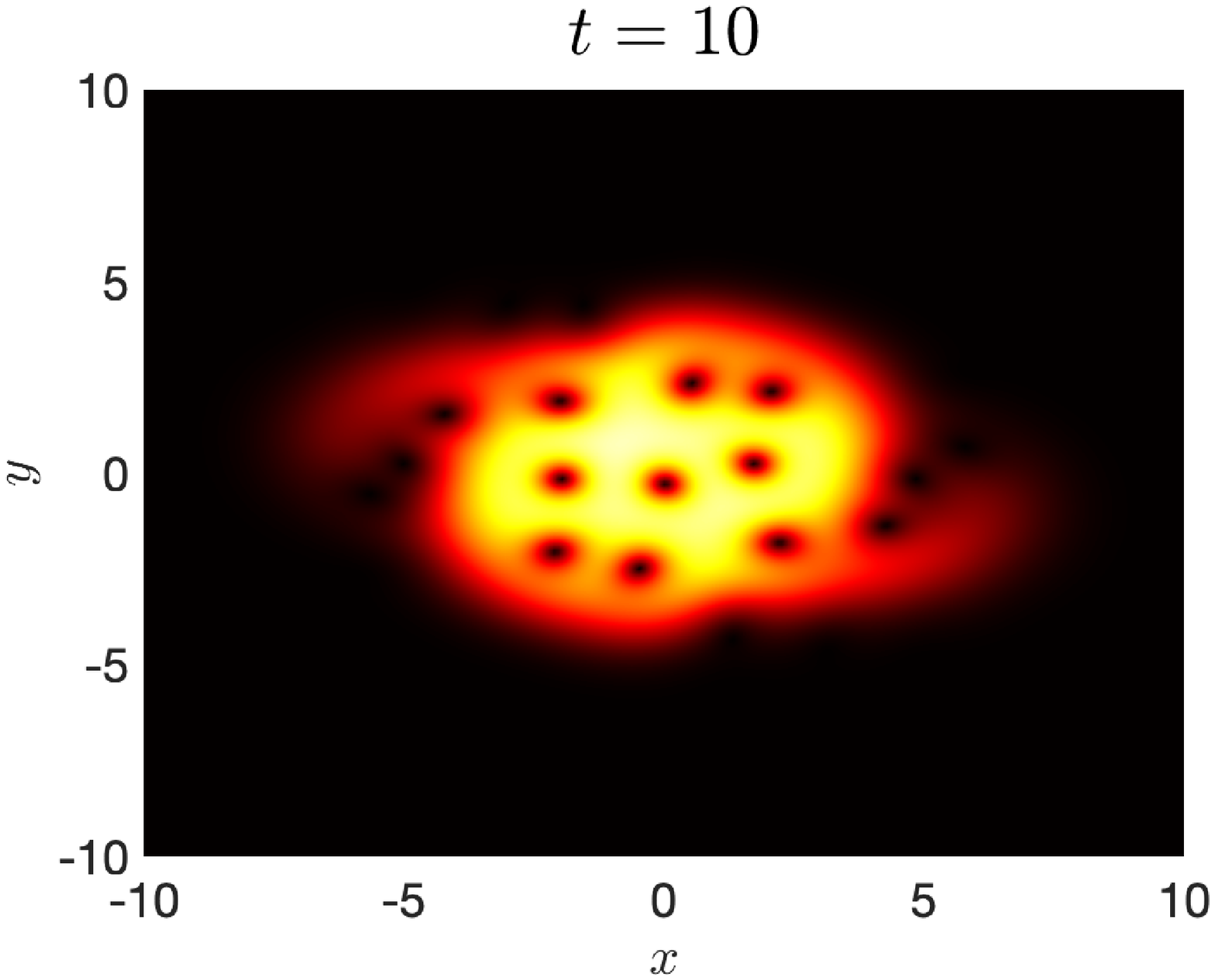}
\end{minipage}
\begin{minipage}[t]{0.3\textwidth}
\centering
\includegraphics[scale=0.27]{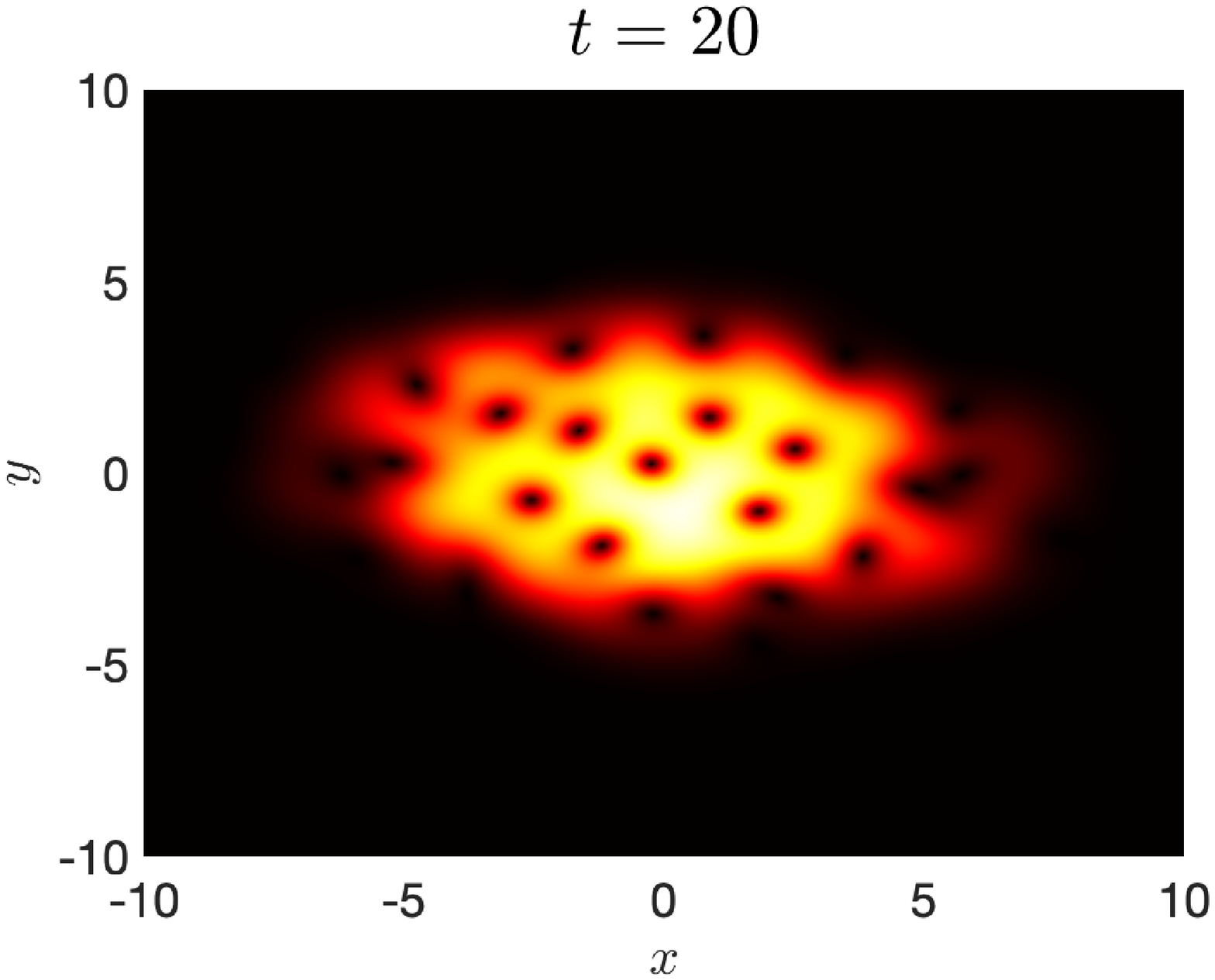}
\end{minipage}
\begin{minipage}[t]{0.3\textwidth}
\centering
\includegraphics[scale=0.27]{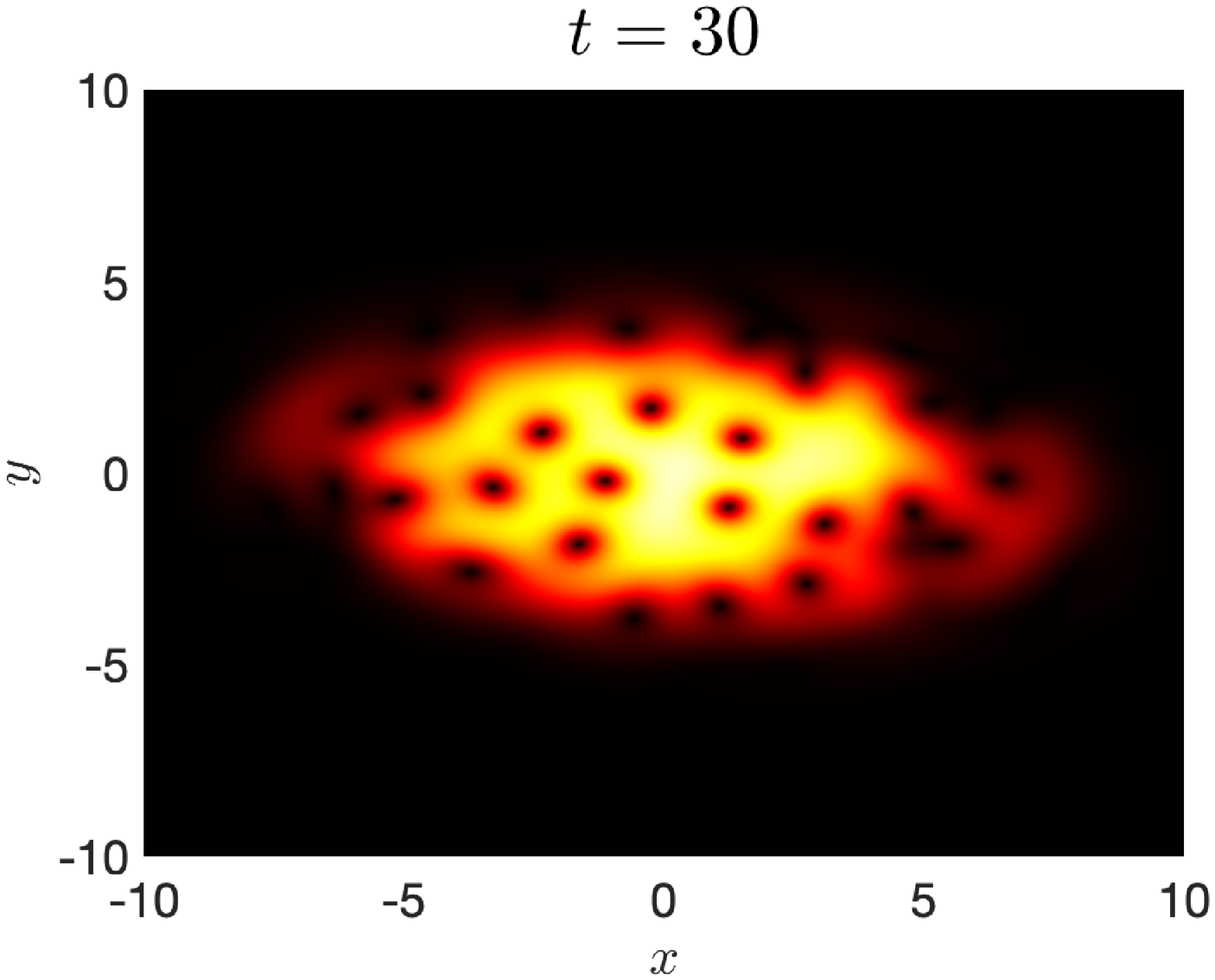}
\end{minipage}
\begin{minipage}[t]{0.3\textwidth}
\centering
\includegraphics[scale=0.27]{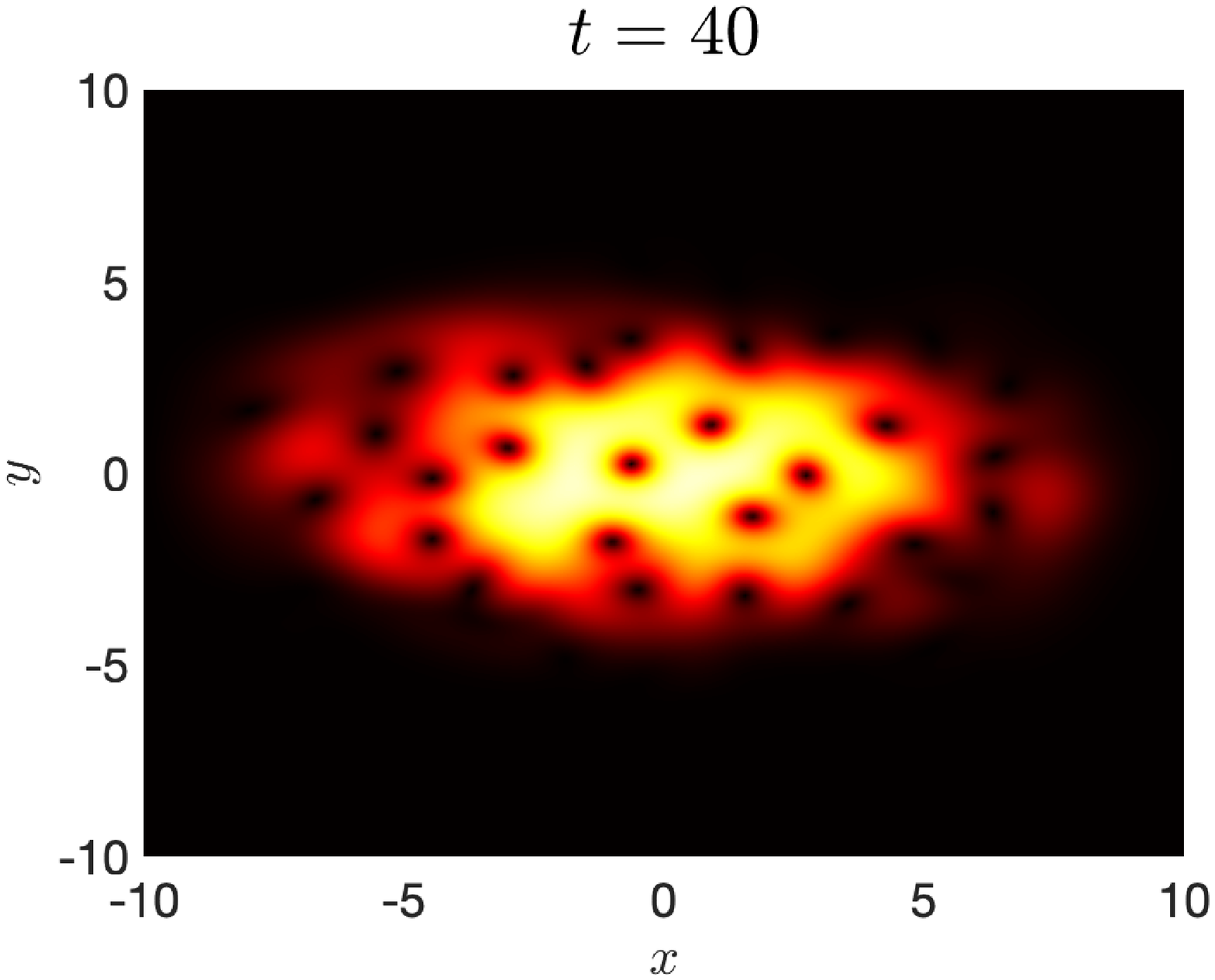}
\end{minipage}
\begin{minipage}[t]{0.3\textwidth}
\centering
\includegraphics[scale=0.27]{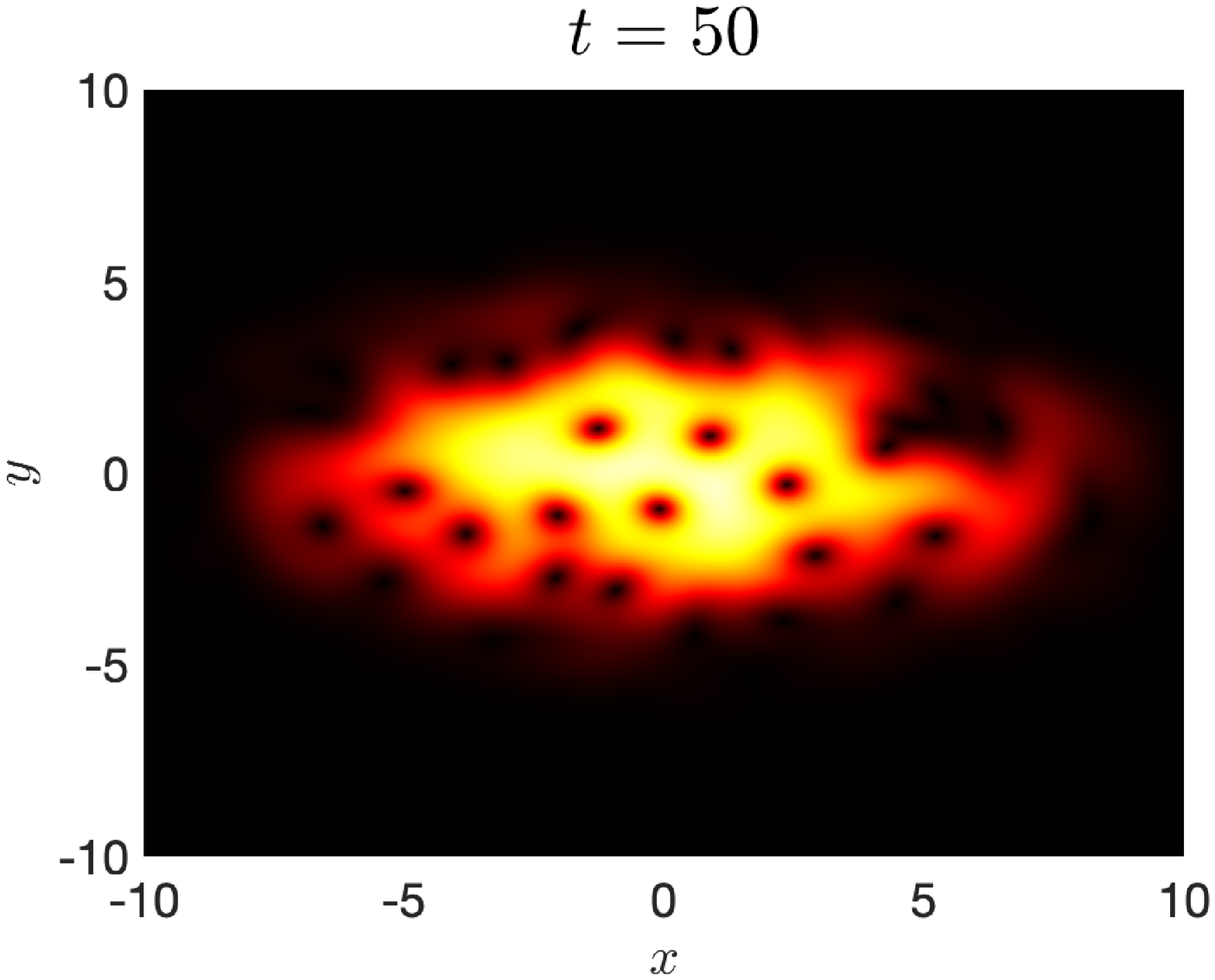}
\end{minipage}
\caption{Absolute value of the numerical solution $|u|$ of the cG($3$)-method applied to the rotating GPE \eqref{numericalproblem}, \eqref{parameter2} on the domain $\D = (-10,10)^2$ with $T = 50$, spatial mesh size $h = 20/256$, and time step size $\tau = 0.5$.} \label{BECsim}
\end{figure}

For the numerical computation 
the spatial discretization parameter is selected as $h=20/256$ which results in 65\,025 degrees of freedom for the space $S_h$ of $\mathcal{P}^1$-Lagrange finite elements on the uniform triangulation. The time stepping is executed with the uniform time step size $\tau = 0.05$ and in each time step the system of nonlinear equations is solved with the fixed point iteration up to a precision of $\varepsilon_{\mathrm{FP}} = 10^{-6}$. \\
Figure \ref{BECsim} shows the numerical result at time instances $t = 0,10,20,30,40,50$. Indeed, the expected dynamics of the condensate are reproduced by the numerical approximation. The essential observation is that the condensate can be tracked in a stable manner by the numerical scheme. This traces to the energy conservation of the cG($q$)-method since non-conserving time integrators are expected to lose the condensate after a certain period of time; see \cite{HenningMalqvist17}. In the dynamics of the simulated BEC we can observe that after the trapping potential is perturbed the vortices within the BEC start to move since the initial ground state is not an equilibrium for the perturbed system. Some of the vortices leave the center part of the BEC and other are coming inside from the region of nearly zero density. Due to the rotation of the BEC a turbulent dynamic can be obtained at the transition region from high density to nearly zero density of the condensate. \\
Summarizing, the cG($q$)-methods presented in this work form a successful high order and energy-conserving time integrator for the GPE. In particular, the numerical experiments show successful approximations of the time evolution of BECs. \\

{\bf Acknowledgements.}
The authors would like to thank the anonymous reviewers for their insightful comments that helped to improve the paper.


\appendix
\sect{Proof of a priori error estimates and uniqueness}

In the first part of the appendix we complete the proof of Lemma \ref{elliptic}:

\begin{proof}[Proof of Lemma \ref{elliptic}]
It is straightforward to show that $(\cdot,\cdot)_\H$ defines a scalar product on $H^1(\D)$ and since $V$ and $b$ are bounded the estimate $\| v \|_{\H} \le C_2 \| v \|_{H^1(\D)}$ follows immediately. 

Using Young's inequality with $\delta > 1$ we have
\begin{align*}
	\int_\D |(\nabla v - \tfrac{\i}{2}b v) |^2 \d \x & \ge \int_\D |\nabla v|^2 - |\nabla v||bv| + \frac{1}{4} |bv|^2 \d \x \ge \int_\D\frac{\delta - 1}{\delta} |\nabla v|^2 + \frac{1 - \delta}{4}|bv|^2 \d \x.
\end{align*}
Choosing $\delta = 1 + \lambda$ with $\lambda$ from \ref{A4} yields for some $C_1 > 0$
\begin{align*}
	\| v \|_\H^2 \ge \frac{\delta - 1}{\delta} \| \nabla v \|^2_{L^2(\D)} + \int_{\D} \Big( V -\frac{\delta}{4} |b|^2 \big) |v|^2 \d x \ge C_1 \| v \|_{H^1(\D)}^2.
\end{align*}
Moreover,
\begin{align*}
	(u,v)_\H & = \int_\D \overline{(\nabla u - \tfrac{\i}{2}b u)} \cdot (\nabla v - \tfrac{\i}{2}b v) + (V - \tfrac{1}{4}|b|^2 )\overline{u} v \d \mathbf{x} \\
	& = \int_\D \overline{\nabla u} \cdot \nabla v \d \x - \frac{\i}{2} \Big( \int_{\D} \overline{\nabla u} \cdot bv \d\x - \int_{\D} b\overline{u} \cdot \nabla v \d \x \Big) + \int_{\D} V \overline{u} v \d \x \\
	& = \int_\D \overline{\nabla u} \cdot \nabla v \d \x - \int_{\D} \i (b \cdot \overline{\nabla u}) v \d \x + \int_{\D} V \overline{u} v \d\x  \\
	& = \langle \H u , v \rangle.
\end{align*}
Hence i) holds. The continuity of $\H$ is obvious and the ellipticity now follows by the first part of the lemma. For the statement ii) we apply elliptic $H^2$-regularity theory (cf. \cite{GilbargTrudinger01}) to the Poisson problem $-\Delta u = \tilde{f}$ with homogeneous Dirichlet boundary condition, where $\tilde{f}:=\H u + \Omega \mathcal{L}_z u -V u \in L^2(\D)$. Then the estimate ii) follows from energy $H^1$-bounds for $u$ together with $\| u \|_{H^2(\D)} \le C \, \| \Delta u \|_{L^2(\D)}$.
\end{proof}

The second part of the appendix is devoted to the a priori error estimates stated in Corollary \ref{aprioriestimates}. The estimates extend the results from \cite{Makridakis99} to the general case of the GPE with $\Omega \neq 0$, $V \neq 0$ and $d \le 3$. We only show the key steps to derive the $L^\infty(I,L^2(\D))$-estimate since the $L^\infty(I,H^1(\D))$-estimate follows the same strategy. The superconvergence is just a slight generalization to the case $d = 3$ of the result in \cite{Makridakis99} and is therefore omitted. So for the $L^\infty(I,L^2(\D))$-error estimates we split the error into
\begin{align*}
	u - \ut_{h,\tau} = u - \I_\tau^\Lo P_h u + e_h, \quad e_h = \I_\tau^\Lo P_h u - \ut_{h,\tau}.
\end{align*}
Then the estimates in \ref{A5} and the interpolation estimate \eqref{LagrangeEstimate} imply for $t \in I_n$
\begin{align} \label{Linftyappendix}
\begin{split}
	\| u - \I_\tau^\Lo P_h u \|_{L^\infty(I_n,L^2(\D))} & \le \| u - \I_\tau^\Lo u \|_{L^\infty(I_n,L^2(\D))} + \| \I_\tau^\Lo u - P_h \I_t^\Lo u \|_{L^\infty(I_n,L^2(\D))} \\
	& \le C \tau_n^{q+1} \| \partial_t^{q+1} u \|_{L^\infty(I_n,L^2(\D))} + C h^{r + 1} \| u \|_{L^\infty(I_n,H^{r+1}(\D))}.
\end{split}
\end{align}
Moreover, we have by the same arguments
\begin{align} \label{L2appendix}
\begin{split}
	\| u - \I_\tau^\Lo P_h u \|_{L^2(I_n \times \D)} & \le \| u - \I_\tau^\Lo u \|_{L^2(I_n \times \D)} + \| \I_\tau^\Lo u - P_h \I_t^\Lo u \|_{L^2(I_n \times \D)} \\
	& \le C \tau_n^{q+1} \| \partial_t^{q+1} u \|_{L^2(I_n \times \D)} + C \tau_n^{1/2} h^{r + 1} \| u \|_{L^\infty(I_n,L^2(\D))}.
\end{split}
\end{align}
So it remains to estimate $e_h$. As before, we introduce the notation $e_h^n : = e_h(t_n)$ and $e_h^{n,i} : = e_h(t^\GL_{n,i})$ for $n = 0,\dots,N-1$, $i = 1,\dots,q$. Then a short computation shows that $e_h$ satisfies the error equation
\begin{align} \label{erroreqappendix}
\begin{split}
	& \sum_{j = 0}^q m_{ij} (e_{h}^{n,j}, \psi) - \i \tau_n w_i^\GL(e_{h}^{n,i},\psi)_\H - \i \beta \int_{I_n} \ell_{n,i}(f(\I_\tau^\Lo P_h u) - f(\ut_{h,\tau}),\psi) \d t \\
	& = (A^{n,i} - \i B^{n,i} - C^{n,i}, \psi) - \i \beta \int_{I_n} \ell_{n,i}(f(\I_\tau^\Lo P_h u) - f(u),\psi) \d t
\end{split}
\end{align}
for all $\psi \in S_h$, $i = 1,\dots,q$ and $n = 0,\dots, N-1$ where
\begin{align*}
	A^{n,i} & := \sum_{j=0}^q w^\Lo_{n,j} \ell_{n,i}'(t^\Lo_{n,j}) u(t^\Lo_{n,j}) -  \int_{I_n} \ell_{n,i}' u \d t, \\
	B^{n,i} & := \sum_{j=0}^q w^\Lo_{n,j} \ell_{n,i}(t^\Lo_{n,j}) \H u(t^\Lo_{n,j}) - \int_{I_n} \ell_{n,i} \H u \d t \\
	C^{n,i} & = \ell_{n,i}(t_{n+1})(u - P_h u )(t_{n+1}) - \sum_{j = 0}^q w_{n,j}^\Lo \ell_{n,i}'(t_{n,j}^\Lo) (u - P_h u)(t_{n,j}^\Lo) - \ell_{n,i}(t_{n})(u - P_h u)(t_n).
\end{align*}
We have the following estimates.

\begin{lemma} \label{ABCest}
Let the assumptions \ref{A1}-\ref{A8} be fulfilled and assume $u,\partial_t u \in L^\infty(I,H^{r+1}(\D))$. Then there are constants $C,\tau_0,h_0 > 0$ such that for all $0 < \tau_n < \tau_0$, $0 < h < h_0$, $i = 1,\dots,q$ and $n = 0,\dots,N-1$ there hold
\begin{align*}
	\| A^{n,i} \|_{L^2(\D)} & \le C \tau_n^{q + \frac{3}{2}} \| \partial_t^{q+2} u \|_{L^2(I_n \times \D)}, \\
	\| B^{n,i} \|_{L^2(\D)} & \le C \tau_n^{q + \frac{3}{2}} \| \partial_t^{q+1} \H u \|_{L^2(I_n \times \D)}, \\
	\| C^{n,i} \|_{L^2(\D)} & \le C \tau_n^{1/2} h^{r+1} \| \partial_t u \|_{L^2(I_n, H^{r+1}(\D))}.
\end{align*}
\end{lemma}

\begin{proof}
The estimates for $A^{n,i}$ and $B^{n,i}$ follow the same way as in the proof of Lemma \ref{ABest}. So we only prove the estimate for $C^{n,i}$. As in the proof of Lemma \ref{consistencyfull} we can find coefficients $\beta_{ij}$ such that
\begin{align*}
	C^{n,i} = \sum_{j = 1}^q \beta_{ij} \big((u - P_hu)(t_{n,j}^\Lo) - (u - P_hu)(t_{n,j-1}^\Lo)\big) = \sum_{j = 1}^q \beta_{ij} \int_{t_{n,j-1}^\Lo}^{t_{n,j}^\Lo} (\partial_t u - P_h \partial_t u)(t) \d t.
\end{align*}
Hence,
\begin{align*}
	\| C^{n,i} \|_{L^2(\D)} \le C \int_{I_n} \| \partial_t u - P_h \partial_t u \|_{L^2(\D)} \d t & \le C h^{r+1} \int_{I_n} \| \partial_t u \|_{H^{r+1}(\D)} \d t \\
	& \le C \tau_n^{1/2} h^{r+1} \| \partial_t u \|_{L^2(I_n,H^{r+1}(\D))}. 
\end{align*}
\end{proof}

Next we have the following local error estimate.

\begin{lemma} \label{localappendix}
Let the assumptions \ref{A1}-\ref{A8} be fulfilled and assume $u,\partial_t u \in L^\infty(I,H^{r+1}(\D))$. Then there are constants $C,\tau_0,h_0 > 0$ such that for all $0 < \tau_n < \tau_0$, $0 < h < h_0$ and $n = 0,\dots,N-1$ there hold
\begin{align*}
	\| e_h \|_{L^2(I_n \times \D)}^2 \le C \tau_n \| e_h^n \|_{L^2(\D)}^2 + \tau_n^2 \mathcal{E}_n
\end{align*}
where
\begin{align*}
	\mathcal{E}_n = C \tau_n^{2q + 2} \Big( \| \partial_t^{q+1} u \|_{L^2(I_n \times \D)}^2 & + \| \partial_t^{q+2} u \|_{L^2(I_n \times \D)}^2 + \| \partial_t^{q+1} \H u \|_{L^2(I_n \times \D)}^2 \Big) \\
	& + C h^{2r+2} \Big( \| \partial_t u \|_{L^2(I_n,H^{r+1}(\D))}^2 + \tau_n \| u \|_{L^\infty(I_n,H^{r+1}(\D))}^2 \Big).
\end{align*}
\end{lemma}

\begin{proof}
Similarly, as in the proof of Lemma \ref{consistencyH1} we introduce $\mathbf{e}_h^{n,j} = (s^\GL_j)^{-\frac{1}{2}} e_h^{n,j}$, $j = 1,\dots,q$. Then testing in \eqref{erroreqappendix} with $\psi = (s^\GL_i)^{-\frac{1}{2}} \mathbf{e}_\tau^{n,i}$, summing over $i = 1,\dots,q$ and taking real parts yields in view of Lemma \ref{stability}
\begin{align*}
	\alpha \sum_{i = 1}^q \| e_h^{n,i} \|_{L^2(\D)}^2 & \le \Re \sum_{i,j = 1}^q \mathbf{m}_{ij} (\mathbf{e}_h^{n,j},\mathbf{e}_h^{n,i}) \\
	& = - \Im \beta \sum_{i = 0}^q \int_{I_n} (s^\GL_i)^{-\frac{1}{2}} \ell_{n,i}(f(\I_\tau^\Lo P_h u) - f(\ut_{h,\tau}),\mathbf{e}_h^{n,i}) \d t \\
	& \quad + \Im \beta \sum_{i = 0}^q \int_{I_n} (s^\GL_i)^{-\frac{1}{2}} \ell_{n,i} (f(\I^\Lo_\tau P_h u) - f(u), \mathbf{e}_h^{n,i}) \d t \\
	& \quad + \Re \sum_{i = 0}^q (s^\GL_i)^{-\frac{1}{2}} (A^{n,i} - \i B^{n,i} - C^{n,i}, \mathbf{e}_h^{n,i}) \\
	& \quad - \Re \sum_{i = 0}^q (s^\GL_i)^{-\frac{1}{2}} m_{i0} \ell_{n,i}(t_n) (e_h^n,  \mathbf{e}_h^{n,i}).
\end{align*}
Estimating the terms on the right hand side using Lemma \ref{cutoff} and Lemma \ref{ABCest} and following the same strategy as in the proof of Lemma \ref{consistencyH1} yields the claim.
\end{proof}

Now we are in the position to prove the a priori error estimates w.r.t. $L^2(\D)$ stated in Corollary \ref{aprioriestimates}.

\begin{proof}[Proof of Corollary \ref{aprioriestimates} ($L^2(\D)$-estimates)]
We test the error equation \eqref{erroreqappendix} with $\psi e_h^{n,i}$, sum over $i = 1,\dots,q$ and take real parts to obtain
\begin{align*}
	\frac{1}{2} \| e_h^{n+1} \|_{L^2(\D)}^2 -  \frac{1}{2} \| e_h^{n} \|_{L^2(\D)}^2 & = \Re \sum_{i = 1}^q \sum_{j = 0}^q m_{ij} (e_h^{n,j},e_h^{n,i}) \\
	& = - \beta\, \Im \sum_{i = 1}^q \int_{I_n} \ell_{n,i} (f(\I_\tau^\Lo P_h u) - f(\ut_{h,\tau}), e_h^{n,i}) \d t \\
	& \quad + \beta\, \Im \sum_{i = 1}^q \int_{I_n} \ell_{n,i} (f(\I_\tau^\Lo P_h u) - f(u), e_h^{n,i}) \d t \\
	& \quad + \Re \sum_{i = 1}^q (A^{n,i} -\i B^{n,i} - C^{n,i}, e_h^{n,i}).
\end{align*}
Similarly, as in the proof of Lemma \ref{semiH1estimate} we estimate the terms on the right hand side using Lemma \ref{cutoff} by
\begin{align*}
	\Big| \beta\, \Im \sum_{i = 1}^q \int_{I_n} \ell_{n,i} (f(\I_\tau^\Lo P_h u) - f(\ut_{h,\tau}), e_h^{n,i}) \d t \Big| \le C \| e_h \|_{L^2(I_n \times \D)}^2
\end{align*}
and with \eqref{L2appendix} by
\begin{align*}
	\Big| \beta\, \Im \sum_{i = 1}^q \int_{I_n} & \ell_{n,i} (f(\I_\tau^\Lo P_h u) - f(u), e_h^{n,i}) \d t \Big| \\
	& \le C \tau_n^{2q+2} \| \partial_t^{q+1} u \|_{L^2(I_n \times \D)}^2 + C \tau_n h^{2r + 2} \| u \|_{L^\infty(I_n,H^{r+1}(\D))}^2 + C \| e_h \|_{L^2(I_n \times \D)}^2 \\
	& \le \mathcal{E}_n +  C \| e_h \|_{L^2(I_n \times \D)}^2.
\end{align*}
The last term is estimated using Lemma \ref{ABCest} by
\begin{align*}
	\Big| \Re \sum_{i = 1}^q (A^{n,i} -\i B^{n,i} - C^{n,i}, e_h^{n,i}) \Big| \le \mathcal{E}_n + C \| e_h \|^2_{L^2(I_n \times \D)}.
\end{align*}
With Lemma \ref{localappendix} we therefore obtain the recursion
\begin{align*}
	\| e_h^{n+1} \|_{L^2(\D)}^2 \le \| e_h^n \|_{L^2(\D)}^2 + C \| e_h \|_{L^2(I_n \times \D)}^2 + \mathcal{E}_n \le (1 + C\tau_n) \| e_h^n \|_{L^2(\D)}^2 + (1 + C\tau_n^2) \mathcal{E}_n
\end{align*}
which yields
\begin{align*}
	\| e_h^n \|_{L^2(\D)}^2 \le (1 + C\tau_n)^n \sum_{m = 0}^{n-1} \mathcal{E}_m \le e^{Ct_{n+1}} \sum_{m = 0}^{n-1} \mathcal{E}_m.
\end{align*}
Thus, we have
\begin{align*}
	\| e_h \|_{L^2(I_n \times \D)}^2 \le C \tau_n \| e_h^n \|_{L^2(\D)}^2 + \tau_n^2 \mathcal{E}_n \le C \tau_n \sum_{m = 0}^n \mathcal{E}_m.
\end{align*}
Now the inverse inequality $\| v \|_{L^\infty(I_n)} \le C \tau_n^{-1/2} \| v \|_{L^2(I_n)}$ for $v \in \mathcal{P}^{q}(I_n)$ implies for all $n = 0,\dots,N-1$
\begin{align*}
	\| e_h \|_{L^\infty(I_n,L^2(\D))} & \le C \tau_n^{-1/2} \| e_h \|_{L^2(I_n \times \D)} \le C \Big( \sum_{m = 0}^n \mathcal{E}_m \Big)^{1/2} \\
	& \le C \tau^{q + 1} \big( \| \partial_t^{q+1} u \|_{L^2(I \times \D)} + \| \partial_t^{q+2} u \|_{L^2(I \times \D)} + \| \partial_t^{q+1} \H u \|_{L^2(I \times \D)} \big) \\
	& \quad + C h^{r + 1} \big( \| \partial_t u \|_{L^2(I,H^{r+1}(\D))} + \| u \|_{L^\infty(I,H^{r+1}(\D))}  \big).
\end{align*}
Finally, thanks to Theorem \ref{mainresult} we have $u_{h,\tau} = \ut_{h,\tau}$ and we obtain
\begin{align*}
	\| u - u_{h,\tau} \|_{L^\infty(I,L^2(\D))} & \le \| u - \I_\tau^\Lo P_h u \|_{L^\infty(I,L^2(\D))} + \| e_h \|_{L^\infty(I,L^2(\D))} \\
	& \le C \tau^{q + 1} \big( \| \partial_t^{q+1} u \|_{L^\infty(I \times \D)} + \| \partial_t^{q+2} u \|_{L^\infty(I \times \D)} + \| \partial_t^{q+1} \H u \|_{L^\infty(I \times \D)} \big) \\
	& \quad + C h^{r + 1} \big( \| \partial_t u \|_{L^2(I,H^{r+1}(\D))} + \| u \|_{L^\infty(I,H^{r+1}(\D))}  \big) \\
	& \le C(u) (\tau^{q+1} + h^{r+1}).
\end{align*}
\end{proof}

Finally we give the proof of the uniqueness for the fully-discrete approximation satisfying \eqref{cGscheme} which is stated in Lemma \ref{uniqueness}.

\begin{proof}[Proof of Lemma \ref{uniqueness}]
We note that the assertion of Lemma \ref{cutoff} still holds true for sufficiently large $C > 0$ when dropping the assumption \ref{A8} and replacing $M$ by $M_0$ (cf. \cite[Lemma 4.1]{Makridakis98}). Thus we can choose the cutoff function $f$ in \eqref{auxproblem} such that the assertions in Lemma \ref{cutoff} are satisfied with $M = M_0$ and such that $u_1, u_2$ solve \eqref{auxproblem}. Let $e := u_1 - u_2$ and $e^n := e(t_n)$, $e^{n,i} := e(t^\GL_{n,i})$ for $n = 0,\dots,N-1$ and $i = 1,\dots,q$. Then we have $e^0 = 0$ and $e$ satisfies
\begin{align} \label{proofuniqueness:eq1}
	\sum_{j = 0}^q m_{ij} ( e^{n,j}, \psi) - \i \tau_n w_i^\GL (e^{n,i}, \psi)_\H = \i \beta \int_{I_n} \ell_{n,i} (f(u_1) - f(u_2), \psi ) \d t
\end{align}
for all $\psi \in S_h$ and $i = 1,\dots,q$ and $n = 0,\dots,N-1$. Now testing in \eqref{proofuniqueness:eq1} with $\psi = e^{n,i}$, summing over $i=1,\dots,q$, and taking real part gives with the same arguments as in the previous sections
\begin{align} \label{proofuniqueness:eq2}
\begin{split}
	\frac{1}{2} \| e^{n+1} \|_{L^2(\D)}^2 - \frac{1}{2} \| e^{n} \|_{L^2(\D)}^2 & = \Re \sum_{i = 1}^q \sum_{j = 0}^q m_{ij} (e^{n,j},e^{n,i}) \\
	& \le \Big| \beta \sum_{i = 1}^q \int_{I_n} \ell_{n,i} (f(u_1) - f(u_2) , e^{n,i}) \d t \Big| \\
	& \le C \int_{I_n} \| f(u_1) - f(u_2) \|_{L^2(\D)} \d t  \left( \sum_{i = 1}^q \| e^{n,i} \|_{L^2(\D)} \right) \\
	& \le C \tau_n^{1/2} \Big( \int_{I_n} \| e \|_{L^2(\D)}^2 \d t \Big)^{1/2}  \Big( \sum_{i = 1}^q \| e^{n,i} \|_{L^2(\D)}^2 \Big)^{1/2} \le C \| e \|_{L^2(I_n \times \D)}^2.
\end{split}
\end{align}
Defining $\mathbf{e}^{n,j} = (s_j^\GL)^{-\frac{1}{2}} e^{n,j}$, $j = 1,\dots,q$ and testing in \eqref{proofuniqueness:eq1} with $\psi = (s_i^\GL)^{-\frac{1}{2}} \mathbf{e}^{n,i}$, summing over $i = 1,\dots,q$ and taking real parts gives with similar arguments as in the previous sections
\begin{align*}
	\Re \sum_{i,j = 1}^q \mathbf{m}_{ij} (\mathbf{e}^{n,j}, \mathbf{e}^{n,i}) & \le \Big| \beta \sum_{i = 1}^q (s_i^\GL)^{-\frac{1}{2}} \int_{I_n} \ell_{n,i} (f(u_1) - f(u_2), \mathbf{e}^{n,i} ) \d t \Big| + \Big| \sum_{i = 1}^q m_{i0} (e^n,\mathbf{e}^{n,i}) \Big| \\
	& \le C \big( \tau_n^{1/2} \| e \|_{L^2(I_n \times \D)} + \| e^n \|_{L^2(\D)} \big) \Big( \sum_{i = 1}^q \| \mathbf{e}^{n,i} \|_{L^2(\D)}^2 \Big)^{1/2}.
\end{align*}
With Lemma \ref{stability} we therefore have shown for some $\alpha > 0$
\begin{align*}
	\alpha \sum_{i = 0}^q \| \mathbf{e}^{n,i} \|_{L^2(\D)}^2 \le C \tau_n \| e \|_{L^2(I_n \times \D)}^2 + C \| e^n \|_{L^2(\D)}^2.
\end{align*}
Hence Lemma \ref{normequiv} now implies
\begin{align*}
	\| e \|_{L^2(I_n \times \D)}^2 \le C \tau_n \sum_{i = 0}^q \| e^{n,i} \|_{L^2(\D)}^2 \le C \tau_n \sum_{i = 0}^q \| \mathbf{e}^{n,i} \|_{L^2(\D)}^2 \le C\tau_n^2 \| e \|_{L^2(I_n \times \D)}^2 + C\tau_n \| e^n \|_{L^2(\D)}^2.
\end{align*}
So for $\tau_0$ sufficiently small we have $\| e \|_{L^2(I_n \times \D)}^2 \le C\tau_n \| e^n \|_{L^2(\D)}^2$ and \eqref{proofuniqueness:eq1} implies
\begin{align*}
	\| e^{n+1} \|_{L^2(\D)}^2 \le (1 + C\tau_n) \| e^{n} \|_{L^2(\D)}^2 \le (1 + C\tau_n)^{n+1} \| e^0 \|_{L^2(\D)}^2 = 0
\end{align*}
for all $n = 0,\dots,N-1$.
\end{proof}

\end{document}